\documentclass[11pt,letterpaper]{article}
\usepackage[margin=1in]{geometry}
\usepackage{url}
\usepackage{graphicx}
\usepackage{amsmath}
\usepackage{amsthm}
\usepackage{mathdots}
\usepackage{amssymb}
\usepackage{verbatim}
\usepackage{tikz}
\usetikzlibrary{cd,matrix,arrows,decorations.pathmorphing}

\newcommand{\C}{\mathbb{C}}
\newcommand{\Z}{\mathbb{Z}}

\newcommand{\calL}{\mathcal{L}} %
\newcommand{\calI}{\mathcal{I}}

\newcommand{\calH}{\mathcal{H}} %

\newcommand{\A}{\mathcal{A}}
\newcommand{\calB}{\mathcal{B}}
\newcommand{\calD}{\mathcal{D}}
\newcommand{\calO}{\mathcal{O}} %

\newcommand{\M}{\mathcal{M}}

\newcommand{\calJ}{\mathcal{J}}
\newcommand{\calF}{\mathcal{F}}

\newcommand{\ve}{\varepsilon}

\newcommand{\bs}{\backslash}

\newcommand{\goth}{\mathfrak}
\newcommand{\into}{\hookrightarrow}

\newcommand{\End}{\textup{End}}

\newcommand{\GL}{\textup{GL}}
\newcommand{\SL}{\textup{SL}}
\newcommand{\SO}{\textup{SO}}
\newcommand{\GSO}{\textup{GSO}}
\newcommand{\GSp}{\textup{GSp}}
\newcommand{\Sp}{\textup{Sp}}
\newcommand{\WO}{\textup{WO}}
\renewcommand{\sp}{\goth{sp}}

\renewcommand{\sl}{\goth{sl}}
\newcommand{\Ad}{\textup{Ad}}
\newcommand{\ad}{\textup{ad}}
\newcommand{\calHom}{\textup{Hom}}

\newcommand{\Ind}{\textup{Ind}}
\newcommand{\ind}{\textup{ind}}
\newcommand{\Tr}{\textup{tr}}

\newcommand{\Frac}{\textup{Frac}}

\newtheorem{Theorem}{Theorem}[section]
\newtheorem{Proposition}[Theorem]{Proposition}
\newtheorem{Conjecture}[Theorem]{Conjecture}
\newtheorem{Theorem/Conjecture}[Theorem]{Theorem/Conjecture}
\newtheorem{Lemma}[Theorem]{Lemma}
\newtheorem{Corollary}[Theorem]{Corollary}
\theoremstyle{remark}
\newtheorem*{remark}{Remark}
\numberwithin{equation}{section}

\begin{document}
\title{The Non-Split Bessel Model on $\GSp(2n)$ as an Iwahori-Hecke Algebra Module}
\date{}
\author{Will Grodzicki}
\maketitle

\begin{abstract}
We realize the non-split Bessel model of Novodvorsky and Piatetski-Shapiro in \cite{NPS} as a generalized Gelfand-Graev representation of $\GSp(4)$, as suggested by Kawanaka in \cite{Kaw}. With uniqueness of the model already established in \cite{NPS}, we establish existence of a Bessel model for unramified principal series representations. We then connect the Iwahori-fixed vectors in the Bessel model to a linear character of the Hecke algebra of $\GSp(4)$ following the method outlined more generally in \cite{BBF2}. We use this connection to calculate the image of Iwahori-fixed vectors of unramified principal series in the model, and ultimately provide an explicit alternator expression for the spherical vector in the model. We show that the resulting alternator expression matches previous results of Bump, Friedberg, and Furusawa in \cite{BFF}. We offer the conjecture that a generalized Bessel model on $\GSp(2n)$ retains the uniqueness property estabilished in the case when $n=2$; assuming that this conjecture holds, we extend all of the previously mentioned results to the case where $n>2$, including existence of the model for unramified principal series representations.


\end{abstract}

\section{Preliminaries} \label{sec:intro}

In their 1973 paper \cite{NPS}, Piatetski-Shapiro and Novodvorsky defined a model for irreducible admissible representations of the group $\GSp(4)$ over a $p$-adic field called the Bessel model, and showed that the dimension of such an embedding is at most 1. Our first task will be to define a generalized Bessel model for irreducible admissible representations of $G = \GSp(2n)$, which we believe possesses the same uniqueness properties as the Bessel model on $\GSp(4)$; we offer this as a formal conjecture later in this section. We then prove that each unramified principal series representation of $\GSp(4)$ has a non-split Bessel model, and we provide an explicit integral representation of the corresponding Bessel functional, which we then generalize to $\GSp(2n)$. Finally, we will proceed to our ultimate goal of providing an explicit expression for the Iwahori-fixed vectors in the model. In particular, when $n=2$, the formula that we develop for the spherical function agrees with the formula for the spherical function in the Bessel model on $\SO(5)$ established by Bump, Friedberg, and Furusawa in \cite{BFF}. Assuming the conjectured uniqueness of the Bessel model on $\GSp(2n)$, we are able to extend these results to rank $n$.


Along the way, we will describe how our construction of the Bessel functional fits into a conjectural program for connecting characters of the Iwahori-Hecke algebra $\calH$ of $G$ and multiplicity-free models of principal series representations. This program, formulated by Brubaker, Bump, and Friedberg in \cite{BBF2}, was motivated by the study of the Whittaker and spherical functionals, which it contains as special cases.

We will show momentarily that the most natural way to view the connection between these models and characters of the Iwahori-Hecke algebra is from the perspective of the ``universal principal series.'' Our description of the universal principal series and its structure as an $\calH$-module follows the treatment provided by Haines, Kottwitz, and Prasad in \cite{HKP}.

Although our results in this paper are with regards to $\GSp(2n)$ over a $p$-adic field, we expect that the methods we use to analyze the Bessel model in this context can be used to analyze other models over other algebraic groups. With this in mind, we will place the following discussion in a more general context. In particular, for this section, let $G$ be a split, connected reductive group over a $p$-adic field $F$ with ring of integers $\goth{o}$ and uniformizer $\pi$. Let $k$ denote the residue field $\goth{o}/(\pi)$, and let $q$ denote its cardinality. Let $B$ be a Borel subgroup of $G$ with maximal torus $T$ and unipotent subgroup $U$ such that $B = TU$. Let $\overline{U}$ denote the opposite unipotent of $U$ in $B$. We assume that these subgroups, as well as $G$, are defined over $\goth{o}$. Note that this means that $K = G(\goth{o})$ is a maximal compact subgroup of $G$. Let $J$ denote the Iwahori subgroup, which is the preimage of $B(k)$ under the canonical homomorphism $G(\goth{o})\to G(k)$. Let $\calH$ denote the Iwahori-Hecke algebra of $G$, which is the $\C$-algebra of functions $C_c(J\bs G/J)$, with multiplication given by convolution.


We define the universal principal series $M$ to be the vector space $C_c(T(\goth{o})U\bs G / J)$. Evidently, we can make $M$ into a right $\calH$-module where $\calH$ acts by convolution. Now, observe that $T/T(\goth{o})$ is isomorphic to the cocharacter group $X_{\ast}(T)$ of $G$ under the map that sends $\mu \in X_{\ast}(T)$ to $\mu(\pi) \in T/T(\goth{o})$. We will write $\mu(\pi)$ as $\pi^{\mu}$ throughout this paper. Define $R := C_c(T/T(\goth{o})) = \C[X_{\ast}(T)]$, and regard $R$ as a left $(T/T(\goth{o}))$-module via the inverse of the ``universal'' character $\chi_{\textup{univ}}: \pi^{\mu} \mapsto \pi^{\mu}$. If we use normalized induction to form $\ind_B^G \chi_{\textup{univ}}^{-1}$, and then take its $J$-fixed vectors $\ind_B^G(\chi_{\textup{univ}}^{-1})^J$, then we can see that $M \simeq \ind_B^G(\chi_{\textup{univ}}^{-1})^J$ as right $\calH$-modules; explicitly we have $\eta: M \to \ind_B^G(\chi_{\textup{univ}}^{-1})^J$ where $$\eta(\phi)(g) = \sum_{\mu \in X_{\ast}(T)} \delta_B(\pi^{\mu})^{-1/2} \pi^{\mu} \phi(\pi^{\mu}g).$$ Here we can see the motivation for our terminology: if we're given an unramified principal series obtained from parabolic induction by a character $\chi: T/T(\goth{o}) \to \C^{\times}$, then $\chi$ determines a $\C$-algebra homomorphism $R \to \C$, and $$\C \otimes_R M \simeq \ind_B^G(\chi^{-1})^J,$$ the Iwahori-fixed vectors of our original unramified principal series.

In order to gain a better understanding of the Hecke algebra $\calH$, we are going to make use of an alternate point of view of $M$. First, we note that $M$ is isomorphic to $\calH$ as a free, rank one right $\calH$-module; it has a $\C$-basis made up of the characteristic functions $1_{T(\goth{o})UwJ}$ where $w$ is an element of the affine Weyl group $\widetilde{W}$. The isomorphism from $\calH$ to $M$ is given by the map $h \mapsto 1_{T(\goth{o})UJ} \ast h$. We can define a left action of $\calH$ on $M$ via this isomorphism: in particular, we identify $h \in \calH$ with the endomorphism $$h: 1_{T(\goth{o})UJ} \ast h' \mapsto 1_{T(\goth{o})UJ} \ast (hh').$$ Using $\eta$, we can transfer this left $\calH$-action to $\ind_B^G(\chi_{\textup{univ}}^{-1})^J$, so that $h \in \calH$ sends $\phi_1 \ast h'$ to $\phi_1 \ast (hh')$, where $\phi_1 = \eta(1_{T(\goth{o})UJ})$. Note that this left action identifies $\calH$ with $\End_{\calH}(M)$. 

Now, if we take the obvious left action of $R$ on $\ind_B^G(\chi_{\textup{univ}}^{-1})^J$ and transfer it via $\eta^{-1}$ to $M$, we see that $R$ embeds into $\End_{\calH}(M)$, and hence embeds into $\calH$. Additionally, the finite Hecke algebra $\calH_0 = C(J\bs K/J)$ is a subalgebra of $\calH$, and there is a vector space isomorphism $\calH \simeq R \otimes_{\C} \calH_0$. While we will often conflate $\pi^{\mu} \in R$ with its embedded image in $\calH$, we would like to point out that the image of $\pi^{\mu}$ is convolution with the characteristic function $1_{J\pi^{\mu}J}$ only when $\mu$ is dominant. We use $T_{s}$ to denote the generator $1_{JsJ}$ of $\calH_0$, where $s$ is a simple reflection in the Weyl group, $W$, of $G$. The generators of $\calH_0$ satisfy the same braid relations that the simple reflections in $W$ satisfy, in addition to satisfying the quadratic relation $$(T_s - q)(T_s + 1) = 0.$$ Finally, to understand $\calH$ in terms of these generators, we need the Bernstein relation, first proved in \cite{Lusz}, which says that, for $\pi^{\mu} \in R$ and $T_s \in \calH_0$, \begin{equation} \label{bernstein} T_s\pi^{\mu} = \pi^{s(\mu)}T_s + (1-q)\frac{\pi^{s(\mu)}-\pi^{\mu}}{1-\pi^{-\alpha^{\vee}}}, \end{equation} where $s = s_{\alpha}$ for a simple root $\alpha$ in the root system $\Phi$ of $G$.

Recall that the spherical function, $\phi^{\circ}$, in $\ind_B^G(\chi_{\textup{univ}}^{-1})^J$ is defined as $$\phi^{\circ}(g) := \delta^{-1/2}(\pi^{\mu})\pi^{-\mu},$$ where $g = tuk$ is the Iwasawa decomposition of $g$ with $u \in U$, $k \in G(\goth{o})$, and $t \in T(F)$ where $t \equiv \pi^{\mu} \in T(F)/T(\goth{o})$ . Using the Iwahori-Bruhat decomposition, we see that $$\phi^{\circ} = \sum_{w \in W} \phi_w,$$ where $\phi_w := \eta(1_{T(\goth{o})UwJ})$. In order to provide an explicit expression for $\phi^{\circ}$ in the Bessel model, we are going to need to use the fact that the spherical function in the model is contained in a submodule isomorphic to $V_{\ve} :=\calH \otimes_{\calH_0} \ve$, where $\ve$ is a linear character of $\calH_0$. 

From the quadratic relation for $\calH_0$, we can see that the only possible eigenvalues for the generators of $\calH_0$ are $-1$ and $q$. The braid relations for $\calH_0$ then imply that we either have two or four linear characters of $\calH_0$, depending on whether or not the Dynkin diagram for $G$ is simply laced. We can see that $V_{\ve} \simeq R$ as vector spaces, so we can transfer the $\calH$-action on $V_{\ve}$ to $R$ via $v_{\ve} \mapsto r$, where $r$ is any element of $R$ and $v_{\ve}$ is the eigenvector of $\calH_0$ corresponding to $\ve$. In practice, our choice of $r$ such that $v_{\ve} \mapsto r$ will be crucial. Roughly stated, a goal of Brubaker, Bump, and Friedberg is to find many examples where, if $\calL$ is an $R$-valued map arising from a unique model, then there is a character $\ve$ of $\calH_0$ and a subgroup $S \subset G$ such that the transformation properties of $\calL$ under $S$ imply that $\calL$ is an $\calH$-map from $M$ to $V_{\ve}$; a key idea here is that the models are connected to the representations of $\calH_0$ via the Springer correspondence - we will discuss this connection further at the end of this section, as well as in Section \ref{sec:springer}.

The simplest examples in this program are the Whittaker and spherical models. If we take $\calL$ to be the $R$-valued spherical functional, uniquely determined up to scalar by the condition that $\calL(k\cdot \phi) = \calL(\phi)$ for all $\phi \in \ind_B^G(\chi_{\textup{univ}}^{-1})$ and $k \in K$, and $\ve$ to be the trivial character on $\calH_0$, then it was shown by Brubaker, Bump, and Friedberg (based on the work of Casselman in \cite{Cas}) that $\calL$ is an $\calH$-intertwiner from $M$ to $V_{\ve}$; in the case where $\calL$ is taken to be the $R$-valued Whittaker functional, uniquely determined up to scalar by the condition that $\calL(\overline{u}\cdot \phi) = \psi(\overline{u})\calL(\phi)$ for all $\phi \in \ind_B^G(\chi_{\textup{univ}}^{-1})$ and $\overline{u} \in \overline{U}$, where $\psi$ is a non-degenerate character of $\overline{U}$, it was shown by Brubaker, Bump, and Licata, in \cite{BBL}, that $\calL$ is an $\calH$-intertwiner from $M$ to $V_{\ve}$, where $\ve$ is the sign character of $\calH_0$. Most recently, in \cite{BBF2}, Brubaker, Bump, and Friedberg showed that the Bessel functional on the doubly-laced group $\SO(2n+1)$ is an $\calH$-intertwiner from $M$ to $V_{\ve}$  in the manner described above; in this case, $\ve$ is the character of $\calH_0$ that acts by $-1$ on long simple roots and by $q$ on short simple roots.


In general, we start with a subgroup $S$ of $G$ and a linear $\C$-valued character $\psi$ of $S$, and we look for an $R$-module homomorphism $\calL: \ind_B^G(\chi_{\textup{univ}}^{-1}) \to R$ such that \begin{equation}\label{equiv} \calL(s \cdot \phi) = \psi(s)\calL(\phi) \textup{ for all $s \in S$ and $\phi \in \ind_B^G(\chi_{\textup{univ}}^{-1})$,}\end{equation} where the action of $G$ on $\ind_B^G(\chi_{\textup{univ}}^{-1})$ is given by right translation. In order to find $\calL$, we will use Mackey theory. In the case where $F$ is a finite field, Mackey theory tells us that the space of $R$-module homomorphisms satisfying \eqref{equiv} is in bijection with the vector space of functions $\Delta: G \to R$ that satisfy the equivariance properties \begin{equation}\label{finitemackey} \Delta(sgb) = \psi(s)\Delta(g)\chi_{\textup{univ}}^{-1}(b) \end{equation} for all $s \in S$, $b \in B$; here we are thinking of $\psi$ as taking values in $R$, since $R$ is a commutative $\C$-algebra with $\C$ included in it. 

When $F$ is a $p$-adic field, Mackey theory tells us that the space of $R$-module homomorphisms satisfying \eqref{equiv} is in bijection with the vector space of \emph{distributions} satisfying \eqref{finitemackey}.\footnote{In practice, for the models that we are considering, any nonzero $\Delta$ satisfying \eqref{finitemackey} is defined on an open set, so that, in these cases, such $\Delta$ are, in fact, functions.} If such a $\Delta$ exists, we get the corresponding $R$-module homomorphism $\calL$ from the convolution $$\calL(\phi)(g) = \int_{B \bs G} \Delta(h^{-1})\phi(hg)\,dh.$$ If such an $\calL$ exists then the space $\Ind_S^G \psi$ is called a model for $\ind_B^G(\chi_{\textup{univ}}^{-1})$ - we say that the model is unique for $\ind_B^G(\chi_{\textup{univ}}^{-1})$ if the space $\calHom_G( \ind_B^G( \chi_{\textup{univ}}^{-1}), \Ind_S^G \psi)$ is one-dimensional, i.e.~if the space of functionals satisfying \eqref{equiv} is one-dimensional.

Based on the formalism of \cite{BBF2}, it can be shown that if $\calL$ is restricted to the space of Iwahori-fixed vectors, $\ind_B^G(\chi_{\textup{univ}}^{-1})^J$, then $\calL$ induces a left $\calH$-module structure on its image. In particular, the group algebra $R$ embedded in $\calH$, as described earlier, acts on the image of $\calL$ by translation. Since $R \simeq \Ind_{\calH_0}^{\calH}\ve$ as vector spaces if $\ve$ is a linear character of $\calH_0$, the following conjecture of Brubaker, Bump and Friedberg is natural:


\begin{Conjecture}[\cite{BBF2}]\label{conjecture}
Let $\calL$ be an $R$-valued linear map on $\ind_B^G(\chi_{\textup{univ}}^{-1})$ obtained from a unique model. Then $\calL$ is an $\calH$-map from $M$ to $V_{\ve} = \Ind_{\calH_0}^{\calH} \ve$ for some choice of linear character $\ve$ of $\calH_0$ and the following diagram commutes:

\begin{equation}\label{diagram}
\begin{tikzpicture}[>=angle 90]
\matrix(a)[matrix of math nodes,
row sep=3em, column sep=4em,
text height=1.5ex, text depth=0.25ex]
{\ind_B^G(\chi_{\textup{univ}}^{-1})^{J} &\\ M \simeq \calH&V_{\ve} \\};
\path[dashed,->,font=\scriptsize]
(a-1-1) edge node[above]{$\calL$}
(a-2-2);
\path[->,font=\scriptsize]
(a-2-1) edge node[left]{$\eta$}
(a-1-1);
\path[->,font=\scriptsize]
(a-2-1) edge node[below]{$\calF_{v}$}
 (a-2-2);
\end{tikzpicture}
\end{equation}

\noindent with $v_{\ve} := \calL(\phi_1)$ and $\calF_{v_{\ve}}: h \mapsto h\cdot v_{\ve}$ where $h$ acts on $v_{\ve}$ according to the module structure on $V_{\ve}$. 
\end{Conjecture}

Of course, such an $\calH$-map $\calL$ is guaranteed to exist since $\calF_{v_{\ve}}$ and $\eta$ are isomorphisms; rather, the dotted line is meant to reiterate the point made earlier that we are looking for a subgroup such that the transformation properties of $\calL$ under this subgroup imply that $\calL$ is an $\calH$-map to $V_{\ve}$. 

One promising set of models that appear to fit into this picture are the ``generalized Gelfand-Graev representations'' introduced by Kawanaka in \cite{Kaw} - these models are classified by nilpotent elements of the Lie algebra of $U$, and the subgroup under which $\calL$ transforms is connected to the associated nilpotent element via the Kirillov orbit method. The particular appeal of this family of representations lies in their conjectured low-multiplicity properties. In particular, we are inspired by Furusawa's use of the Bessel model on $\SO(2n+1)$ in his construction of the standard $L$-function on $\SO(2n+1) \times \GL(n)$ in Section 6 of \cite{BFF} and believe that we will be able to use these models to construct new integral representations of $L$-functions. We provide a more detailed description of generalized Gelfand-Graev representations and their connection to the program described above in Section \ref{sec:gggr}.


In this paper, we will realize the Bessel model on $\GSp(2n)$ as a generalized Gelfand-Graev model in Section \ref{sec:besselmodel}, and, assuming that $\ind_B^G (\chi_{\textup{univ}}^{-1})$ embeds uniquely into the model, we will show in Section \ref{sec:intertwiner} that the associated Bessel functional provides another example of an $\calH$-map $\calL$ as described in Conjecture \ref{conjecture}. We now explicitly state the uniqueness assumption that we are placing on the model:

\begin{Theorem/Conjecture}\label{conjecture2}
There is a unique embedding of $\ind_B^G (\chi_{\textup{univ}}^{-1})$ into the generalized Bessel model, $\Ind_{\overline{U}_AZ_L}^G \widetilde{\psi}_A$ (as defined in Section \ref{sec:bessel}).
\end{Theorem/Conjecture}

This uniqueness condition was verified in the rank 2 case in \cite{NPS}; we only need it for the Theorem to hold for rank $n$. We will provide existence in the rank $n$ case in Section \ref{sec:mackey}. While the symplectic version of the Bessel model has sat untouched since \cite{NPS}, the Bessel model in the odd-orthogonal case was proved to satisfy this uniqueness condition in \cite{GGP} (Corollary 15.3). Indeed, in addition to the conjecture above, we suspect that the Bessel model for $\GSp(2n)$ ($n>2$) has a similar multiplicity one property. 

Our main theorem in this paper is the following:

\begin{Theorem}\label{thm:main}
Let $G = \GSp(2n)$ and let $\ve$ be the character of $\calH_0$ that acts by multiplication by $-1$ on long simple roots and acts by $q$ on short simple roots. Let $V_{\ve} = \Ind_{\calH_0}^{\calH} \ve$. Then, assuming Theorem/Conjecture \ref{conjecture2}, the diagram \eqref{diagram} commutes by taking $v_{\ve} = \pi^{\rho_{\ve}^{\vee}}$, where $\rho_{\ve}$ is half of the sum of the long positive roots; and by taking $\calL= \calB$, the non-split Bessel functional (originally defined on $\GSp(4)$ by Piatetski-Shapiro and Novodvorsky).
\end{Theorem}

It should be noted that the split Bessel model should also give rise to a functional fitting into Conjecture \ref{conjecture} - however, in this case we suspect that one can show that this model is related to the sign character of $\calH_0$.


As mentioned above, before we prove this theorem, we will discuss our generalization of the Bessel model of Novodvorsky and Piatetski-Shapiro from $\GSp(4)$ to $\GSp(2n)$ in Section \ref{sec:besselmodel}, and then we will use Mackey theory to prove the existence of a Bessel model for $\ind_B^G (\chi_{\textup{univ}}^{-1})$ in Section \ref{sec:mackey}. We conclude Section \ref{sec:mackey} with an explicit realization of the Bessel functional as an integral. And then, once we prove Theorem \ref{thm:main} in Section \ref{sec:intertwiner}, we will use that result in Section \ref{sec:besspherical} to calculate the images of the Iwahori-fixed vectors $\{\phi_w\}_{w \in W}$ on torus elements in the model $V_{\ve}$, which has not previously appeared in the literature, even for $n=2$. In particular, we prove the following theorem:


\begin{Theorem}\label{thm:iwahori}
For dominant $\lambda$ and fixed $w$, $$\calB(\pi^{-\lambda}\cdot \phi_w) = \frac{1}{m(J\pi^{\lambda}J)} T_w \pi^{\lambda} \cdot v_{\ve},$$ where the action of $T$ on $\ind_B^G(\chi_{\textup{univ}}^{-1})$ is by right translation and where the action of $T_w\pi^{\lambda}$ on $v_{\ve}$ is the left action on $v_{\ve}$ appearing in the definition of $\calB$.
\end{Theorem}

Using Theorem \ref{thm:iwahori}, we will also be able to calculate the image of the spherical function in the model, which, in the case when $n=2$, gives a new proof of the same result from \cite{BFF} (in what follows, let $\Phi^+$ denote a choice of positive roots of $\Phi$):


\begin{Theorem}\label{thm:spherical}
Let $\rho$ be the half-sum of the positive roots of $\Phi$, and let $\rho_{\ve}$ be as defined in Theorem \ref{thm:main}. Then, for any dominant coweight $\lambda$, $$\calB(\pi^{-\lambda}\cdot \phi^{\circ}) = \frac{\pi^{-\rho_{\ve}^{\vee}}\displaystyle\prod_{\alpha \in \Phi^+,\, \alpha \textup{ long}} (1-q\pi^{\alpha^{\vee}})} {\pi^{\rho^{\vee}}\displaystyle\prod_{\alpha \in \Phi^+} (1-\pi^{-\alpha^{\vee}})} \A\left(\left( \prod_{\alpha \in \Phi^+,\,\alpha \textup{ short}} (1-q\pi^{\alpha^{\vee}})\right) \pi^{2\rho_{\ve}^{\vee}-\rho^{\vee}+\lambda}\right),$$ where the action of $T$ on $\ind_B^G (\chi_{\textup{univ}}^{-1})$ is by right translation and where $\A$ denotes the standard alternator expression $\A(\pi^{\mu}) = \sum_{w \in W}(-1)^{\ell(w)}w\pi^{\mu}$ with $W$ acting on $X_{\ast}(T)$ in the usual way.
\end{Theorem}

We note here that our proof of Theorem \ref{thm:main} does not rely on prior knowledge of the image of the spherical function in the model - in this way our method of proof differs from the proofs of similar results in \cite{BBF2}. Instead, we will calculate the relevant intertwining constants directly. 

In Section \ref{sec:womodels}, we discuss the fourth character, $\sigma$, of the finite Hecke algebra of $\GSp(2n)$, which acts by multiplication by $q$ on long simple roots and $-1$ on short simple roots, specifically in the case where $n=2$. At this time we do not have a realization of the intertwiner $\calL$ satisfying the diagram \eqref{diagram} that is also defined according to a subgroup transformation, but we have matched the image of the spherical function under $\calF_{v_\sigma}$ for $\sigma$ to the image of the spherical function in the Whittaker-Orthogonal models defined by Bump, Friedberg and Ginzburg in \cite{BFG}. In particular, we prove the following proposition:

\begin{Proposition}\label{prop:shalika}
Let $WO$ be the Whittaker-Orthogonal functional on an unramified principal series representation $\tau$ of $\SO(6)$, such that $\tau$ is a local lifting of an unramified principal series representation of $\Sp(4)$. Then $\calF_{v_{\sigma}}(\pi^{-\lambda}\cdot 1_{T(\goth{o})UK})$ and $\WO(z^{-\lambda}\cdot \phi^{\circ})$ agree, for any dominant coweight $\lambda$.
\end{Proposition}

Piatetski-Shapiro and Novodvorsky do not provide an explicit integral formula for their functional, so part of our task in proving Theorem \ref{thm:main} is coming up with the correct integral formula for $\calB$. Our method for doing this follows what we believe to be the general method for connecting models of the form $\Ind_S^G \psi$ to characters of $\calH_0$. We will say a bit about this 
in the next section before moving on to the main sections of the paper, which will be focused on the theorems mentioned above. In Section \ref{sec:springer}, we will give further details on this conjectured construction of unique models for characters of $\calH_0$.

We thank Ben Brubaker for many helpful conversations and communications.

\section{Generalized Gelfand-Graev Representations} \label{sec:gggr}


Let $G$ be as in Section \ref{sec:intro}. With notation carried over from Section \ref{sec:intro}, we will let $\goth{g}$ denote the Lie algebra of $G$, and $\goth{u}$ denote the Lie algebra of $U$. Let $f$ denote the bijective $F$-morphism from $U$ to $\goth{u}$.\footnote{Explicit choices of the ``Springer's morphism'' $f$ for classical type and exceptional type are given in Section 1.2 of \cite{Kaw2}.} Following Yamashita in \cite{Yam}, we let $\theta$ where $^{\theta}X = -X^{\top}$ denote the Cartan involution of $\goth{g}$, and let $\goth{u}^{\ast}$ denote the dual space of $\goth{u}$. Then, for $X \in \goth{u}$, we define $X^{\ast} \in \goth{u}^{\ast}$ by \begin{equation} \langle X^{\ast}, Y \rangle = B(Y, {^{\theta}X}), \quad \textup{for } Y \in \goth{u},\end{equation} where $B$ denotes the Killing form of $\goth{g}$.

We believe that the unique models that give rise to an $R$-homomorphism $\calL$ as described in Section \ref{sec:intro} are related to Kawanaka's construction of the ``generalized Gelfand-Graev representations" of $G$ (gGGr) in \cite{Kaw}. Although Kawanaka's results are given in the context of finite groups of Lie type, we believe that they can be suitably adapted for the $p$-adic setting. 

To construct a gGGr, we begin with a nilpotent $\Ad(G)$-orbit in $\goth{u}$ with representative $A$. One can define a $\Z$-grading of $\goth{g}$ according to $A$, \begin{equation} \goth{g} = \bigoplus_i \goth{g}(i)_A, \end{equation} such that $A \in \goth{g}(2)_A$, $\goth{p}_A =  \oplus_{i\geq 0} \goth{g}(i)_A$ is the Lie algebra of a parabolic subgroup $P_A$ of $G$, and $\goth{u}_{i,A} = \oplus_{j\geq i} \goth{g}(j)_A$ ($i\geq 1$) is the Lie algebra of the unipotent subgroup $U_{i,A}$ of $P_A$. Note that $\goth{u}_{i,A}^{\ast}$ can be identified with $\oplus_{j\geq i} \goth{g}(-j)_A$ via $\langle \cdot, \cdot \rangle$. We then use Kirillov's orbit method to form the attached representation $\eta_A$ on $U_A = U_{1,A}$ - this is done by taking the character \begin{equation}\xi_{A}(\exp(Y)) = \xi_0(\langle A^{\ast},Y\rangle),\quad Y \in \goth{u}_{2,A}\end{equation} defined on $U_{2,A}$ and extending it to a character of an intermediate subgroup before inducing to $U_{A}$ (here $\xi_0$ is a non-trivial additive character of $F$).

Let $L_A$ denote the Levi subgroup of $P_A$. This subgroup acts on $U_A$ via conjugation, and hence acts on the unitary dual $\widehat{U}_A$ of $U_A$ via \begin{equation} \ell \cdot [\eta] = [\ell \cdot \eta], \quad (\ell \cdot \eta)(u) = \eta(\ell^{-1}u\ell) \quad (u \in U_A),\end{equation} where $\ell \in L_A$ and $[\eta] \in \widehat{U}_A$ is the equivalence class of the irreducible representation $\eta$ of $U_A$. 

We denote by $Z_L(\eta_A)$ the stabilizer subgroup of the equivalence class of the representation $\eta_A$ in $L_A$. As the following lemma shows, this subgroup is equal to the centralizer, $Z_L(A)$, of $^{\theta}A$ in $L_A$:

\begin{Lemma}[\cite{Yam}, Lemma 2.1]\label{stabilizer}
The subgroup $Z_L(\eta_A)$ coincides with $Z_L(A)$.
\end{Lemma}
\begin{proof}
By the Ad-invariance of the Killing form, we can see that $\ell^{-1} \cdot [\eta_A] = [\eta_{\Ad( ^{\theta}\ell)A}]$ for any $\ell \in L_A$. If we let $\nu$ denote the Kirillov correspondence $\nu: \goth{u}_A^{\ast}/U_A \to \widehat{U}_A$, then the previous statement is equivalent to the statement $\ell^{-1}\cdot \nu([A^{\ast}]) = \nu([(\Ad( ^{\theta}\ell)A)^{\ast}])$, where $[X^{\ast}]$ denotes the $\Ad^{\ast}(U_A)$-orbit through $X^{\ast}$ in $\goth{u}_A$. Thus, $\ell^{-1}$ (and hence $\ell$) is in $Z_L(\eta_A)$ if and only if $[A^{\ast}] = [(\Ad( ^{\theta}\ell)A)^{\ast}]$. The result will follow if we can show that \begin{equation} \label{eqn:orbit} [A^{\ast}] = [(\Ad( ^{\theta}\ell)A)^{\ast}] \textup{ if and only if } ^{\theta}A = \Ad(\ell) ({^{\theta}A}).\end{equation}

In order to prove this final statement, we first show that \begin{equation}\label{eqn:coadjoint} [X^{\ast}] = X^{\ast} + \goth{g}(-1)_A \textup{ for any } X \in \goth{g}(2)_A,\end{equation} where we are thinking of $\goth{g}(-1)_A$ as being identified with the subspace of $\goth{u}^{\ast}$ consisting of elements that vanish on $\goth{u}_{2,A}$ (\eqref{eqn:coadjoint} is Lemma 1.2.4~in \cite{Kaw2}). The identity \eqref{eqn:coadjoint} is essentially a consequence of the identity $\Ad(u)X = f^{-1}(\ad f(u))X$, where $u\in U_A$ and $X \in \mathfrak{g}$; in order to prove \eqref{eqn:coadjoint}, it will be useful to rewrite the previous identity as in Lemma 1.2.1~in \cite{Kaw2}: \begin{equation}\label{eqn:kaw} \Ad(u)X - (X+d[f(u),X]) \in \bigoplus_{\ell \geq 2i+j} \goth{g}(\ell)_A, \end{equation} where $u \in U_{i,A}$, $X \in \goth{g}(j)_A$, and $d \in F-\{0\}$. Now, if $X \in \goth{u}_A$, then, using the identification of $\goth{u}_A^{\ast}$ with $\oplus_{i>0} \goth{g}(-i)$ and the $\Ad$-invariance of the Killing form, we see that $\Ad^{\ast}(u)X^{\ast} = p(u^{-1}X^{\ast}u)$, where $p$ denotes projection onto $\oplus_{i>0} \goth{g}(-i)_A$.\footnote{The projection map shows up here because $B(Y,W) = 0$ if $Y \in \goth{u}_A$ and $W \in \oplus_{i\geq 0} \goth{g}(i)_A$.} Putting the preceding discussion together with \eqref{eqn:kaw}, we see that $$[X^{\ast}] \subset X^{\ast}+ \goth{g}(-1)_A \textup{ if } X \in \goth{g}(2)_A.$$ It remains to show that this containment is actually an equality.

By Theorem 2 in \cite{Ros}, we know that $[X^{\ast}]$ is closed, so to prove \eqref{eqn:coadjoint} it suffices to check the dimensions of each side. To find $\dim [X^{\ast}]$, we first note that \eqref{eqn:kaw} implies that $\{g \in G \mid g^{-1}X^{\ast}g = X^{\ast}\} \subset \overline{P}_A$, where $\overline{P}_A$ is the opposite parabolic associated to $A$. Thus, if $u \in U_A$ such that $u^{-1}X^{\ast}u = X^{\ast}$, then $u$ is the identity element. Now, since $\Ad^{\ast}(u)X^{\ast} = p(u^{-1}X^{\ast}u)$, we see that, if $u \in U_A$, then $\Ad^{\ast}(u)X^{\ast} = X^{\ast}$ if and only if $[f(u),X^{\ast}] = 0$ or $f(u) \in \mathfrak{u}_{2,A}$ (i.e.~$u \in U_{2,A}$). But, if $[f(u),X^{\ast}] = 0$, then $u^{-1}X^{\ast}u = X^{\ast}$, and hence $u$ is the identity. Hence, $$U_{2,A} = \{ u \in U_A \mid \Ad^{\ast}(u)X^{\ast} = X^{\ast}\},$$ which tells us that $$\dim [X^{\ast}] = \dim U_A - \dim U_{2,A} = \dim \goth{g}(1)_A = \dim (X^{\ast} + \goth{g}(-1)_A),$$ and we have proved \eqref{eqn:coadjoint}. 



Using \eqref{eqn:coadjoint}, we see that $[A^{\ast}] = [(\Ad(^{\theta}\ell)A)^{\ast}]$ if and only if $A^{\ast} = (\Ad(^{\theta}\ell)A)^{\ast}$; this last equality holds if and only if $\ell \in Z_L(A)$, proving \eqref{eqn:orbit}.




\end{proof}


It is natural to extend $\eta_A$ to a representation of $U_A \rtimes Z_L(A)$; our next step, then, is to build a representation $\widetilde{\eta}_{A,\alpha}$ on $U_A\rtimes Z_L(A)$ by taking the tensor product of $\eta_A$ with a representation $\alpha$ of $Z_L(A)$. For each irreducible representation $\alpha$ of $Z_L(A)$, we say that the gGGr associated to the pair $(A,\alpha)$ is $\Gamma_{A,\alpha} := \Ind_{U_AZ_L(A)}^G \widetilde{\eta}_{A,\alpha}$. If the group $G$ is defined over a finite field instead of a $p$-adic field, Kawanaka offers a method of producing gGGr's that contain each unipotent representation with multiplicity one (Conjecture 2.4.5 in \cite{Kaw}). Since the principal series representations are precisely those representations containing a $B$-fixed vector, Kawanaka's conjecture implies that, for each irreducible $\calH_0$-module, there should be a unique gGGr containing it with multiplicity one. Kawanaka's notes after the conjecture suggest that the nilpotent element $A$ used in the construction of a gGGr $\Gamma_{A,\alpha}$ and the irreducible representation of $\calH_0$ contained inside the $B$-fixed vectors of $\Gamma_{A,\alpha}$ are linked via the Springer correspondence. 


Shifting back to the $p$-adic setting, we note that, in \cite{MW}, M\oe glin and Waldspurger give a treatment of those representations - also referred to as gGGr's in \cite{Kaw} - that are constructed by inducing $\eta_A$ from $U_A$ up to $G$ directly. However, one of our goals is to find useful models - for example, as mentioned in Section \ref{sec:intro}, we expect that the gGGr's (as defined in the previous paragraph) will find applications in the construction of integral representations of $L$-functions, in a sense similar to the application of the Bessel model on $\SO(2n+1)$ discussed in Section 6 of \cite{BFF} - and gGGr's of the form $\Gamma_A := \Ind_{U_A}^G \eta_A$ will not have the low-multiplicity properties that we desire. The idea, then, is to decompose $\Gamma_A$ into a direct sum of  gGGr's of the form $\Gamma_{A,\alpha}$ which will have the desired low-multiplicity properties. In the finite field setting, this is exactly what happens, since the stabilizer $Z_L(A)$ is reductive. It is also true that $Z_L(A)$ is reductive when $G$ is defined over a $p$-adic field; we state this result without proof:


\begin{Lemma}[\cite{Car}, Proposition 5.5.9]
The subgroup $Z_L(A)$ is reductive.
\end{Lemma}

A proof of this result can be found in Section 5.5 in \cite{Car}. It should be noted that Carter's proof is given for $G$ defined over an algebraically closed field, and relies on a proof of the Jacobson-Morozov Lemma given in this context. That the Jacobson-Morozov Lemma holds over a field of characteristic 0 seems to be a well-known result (cf.~Section 2.4 in \cite{KL}), and a proof of a closely-related result can be found in Section 8 of \cite{BT} (more recently, a proof of this exact result can be found in Section 2 of \cite{Wit}). The rest of Carter's proof applies to this context without alteration.

However, in contrast to what we observe in the finite field setting, the representation $\widetilde{\eta}_{A,\alpha}$ is not necessarily guaranteed to be a genuine representation if $G$ is instead defined over a $p$-adic field; in general, we are only guaranteed that it is a projective representation of $U_A\rtimes Z_L(A)$. With that said, if $\eta_A$ is a character and $\widetilde{\eta}_{A,\alpha}$ is formed by tensoring with a character of $Z_L(A)$ - as is the case for the Bessel model on $\GSp(4)$ - then $\widetilde{\eta}_{A,\alpha}$ will be a genuine representation.



Unlike the Whittaker model, which served as the inspiration for the definition of a Gelfand-Graev representation (see \cite{GG}), the spherical model and Bessel model are not realized directly as gGGr's. Instead, we realize these models by extending $\eta_A$ from $U_A$ to $U_A\rtimes (Z_L(A)\cap G(\goth{o}))$, and then inducing to $G$. Note that this choice to induce from $Z_L(A) \cap G(\goth{o})$ means that the central character of a given representation will not play a role in whether or not that representation appears in the model. We also note that, in the case of the Whittaker model, $Z_L(A)$ is trivial, so it appears that this method of extending $\eta_A$ to the semidirect product of $U_A$ and $Z_L(A) \cap G(\goth{o})$ is a step towards understanding the general construction of gGGr's over local fields. As mentioned in Section \ref{sec:intro}, in Section \ref{sec:springer} we will expand on the conjectured connection between nilpotent orbits and unique models for characters of the Hecke algebra. 



\section{The Bessel Model and the Bessel Functional} \label{sec:bessel}

We return now to the setting where $G = \GSp(2n)$, and show how the Bessel model as formulated in \cite{NPS} fits into the narrative described in Section \ref{sec:intro} before we move on to establishing our main results. We carry all of our notation through from the previous section. We will have need to realize specific elements of $G$, and so we will explicitly define $G$ as $$G:= \{g \in M_{2n}(F) \mid g^{\top}\Omega g = k\Omega, k \in F^{\times}\},$$ where $$\Omega = \begin{pmatrix} & -\Omega' \\ \Omega' & \end{pmatrix}$$ and $\Omega'$ is the $n\times n$ matrix with 1's on the antidiagonal. As in Section \ref{sec:intro}, we let $\Phi$ denote the root system of $G$, with short simple roots $\alpha_1,\hdots, \alpha_{n-1}$ and long simple root $\alpha_n$. Let $s_1,\hdots, s_n$ and $w_0$ denote the corresponding simple reflections and long element, respectively, in $W$. Let $\rho$ denote the half-sum of the positive roots of $\Phi$, and let $\Phi^+$ and $\Phi^-$ denote the sets of positive and negative roots of $\Phi$, respectively.

\subsection{The Bessel Model as a Generalized Gelfand-Graev Representation} \label{sec:besselmodel}

The transformation property satisfied by the Bessel model depends on the parabolic subgroup $P_A$ of $G$ containing the subgroup corresponding to the negative short simple roots. We can factor $P_A = L_AU_A$ where $L_A$ is the Levi component of $P$, and $U_A$ is the unipotent component of $P$, as described in Section \ref{sec:gggr}. In this case, the nilpotent element $A$ can be chosen so that $A$ is the sum of non-zero elements in the subalgebra $\sum \alpha$, where the sum is taken over the long roots in $\Phi^+$. Let $\overline{U}_A$ denote the opposite unipotent of $U_A$. Let $\psi_0$ be a non-degenerate additive character on $F^+$, and let $\psi_A(u) = \psi_0(\Tr(ru'))$ for $u \in \overline{U}_A$, where $u'$ is the lower left $(n\times n)$-block of $u$ and $r$ is the upper right $(n\times n)$-block of $A$. The linear character $\psi_A$ is the representation of $\overline{U}_A$ that we denoted as $\eta_A$ in Section \ref{sec:gggr}.



We wish to extend $\psi_A$ to a character, $\widetilde{\psi}_A$, of $\overline{U}_A \rtimes Z_L(\psi_A)$, where $Z_L(\psi_A)$ is the stabilizer of the equivalence class of $\psi_A$. From Lemma \ref{stabilizer}, we know that $Z_L(\psi_A) = Z_L(A)$, where $Z_L(A)$ is the centralizer of $A$ in $L_A$. We choose $A$ so that 
$$r = \begin{pmatrix} & & & -\omega_{n-1} \\ & & \iddots & \\ & -\omega_1 & &  \\ 1 & & & \end{pmatrix}.$$ In order to have a unique model in the rank $n$ case, it is likely that we will need to have some sort of condition on $\omega_1,\hdots \omega_{n-1}$, much like we do in the case where $G = \SO(2n+1)$ (cf.~the discussion of Bessel models in Section 1 of \cite{BFF}). Indeed, we see such a condition arise already in the rank 2 case - namely, that $\omega_1 \in F^{\ast}\bs (F^{\ast})^2$ (cf.~the proof of Theorem \ref{thm:functional}).

From our choice of $A$, we see that $Z_L = Z_L(A)$ is the subgroup of $L_A$ with $\GSO(n)$ blocks on the diagonal according to the symmetric bilinear form $$\begin{pmatrix} -\omega_{n-1} & & & \\ & \ddots & & \\ & & -\omega_1 & \\ & & & 1 \end{pmatrix}.$$ We pause here to note that the simple reflections $s_1,\hdots, s_{n-1}$, corresponding to the short simple roots, are contained in $Z_L$. We will denote the subgroup of $W$ generated by these simple reflections as $W_L$, and we define $$\Phi_L := \left\{\alpha \in \Phi \mid \alpha = \sum_{i=1}^{n-1} c_i\alpha_i\right\}.$$ Note that $\alpha \in \Phi_L$ if and only if the root subgroup corresponding to $\alpha$, denoted $x_{\alpha}(F)$, is a subgroup of $L$. Finally, we define $\Phi_L^+ := \Phi_L \cap \Phi^+$, and we define $\Phi_L^-$ analogously.


Turning our attention back to our realization of the Bessel model as a gGGr, we define $\widetilde{\psi}_A(ut) = \psi_A(u)$ for $u \in \overline{U}_A$ and $t \in Z_L$. Note that this representation is the one that would be denoted by $\widetilde{\eta}_{A,1}$ in the previous section, constructed from $\psi_A$ and the trivial representation of $Z_L$. Then, following the previous section, we define the Bessel model to be $\Ind_{\overline{U}_AZ_L}^G(\widetilde{\psi}_A)$.

The Bessel functional for an irreducible admissible representation $\theta$ on $G$ is defined to be a linear functional $\calB$ on the representation space $V_{\theta}$ of $\theta$ such that $$\calB(\theta(ut) v) = \widetilde{\psi}_A(ut)\calB(v),$$ for $v \in V_{\theta}$, $t \in Z_L$ and $u \in \overline{U}_A$. In particular, note that this means that $\widetilde{\psi}_A$ must agree with the central character of $\theta$. Following \cite{BFF}, let $Z_L(\goth{o}) = Z_L \cap \SL(2n,\goth{o})$, so that $Z_L$ is the semidirect product of the compact group $Z_L(\goth{o})$ and the center of $G$. We want the character $\widetilde{\psi}_A$ to have $\goth{o}$ as its conductor, so we choose $r \in \textup{Mat}(n,\goth{o})$. We will discuss questions of existence and uniqueness of a Bessel functional for $\ind_B^G(\chi_{\textup{univ}}^{-1})$ in depth in the next section. We end this section with the statement of Novodvorsky and Piatetski-Shapiro's theorem regarding the (more general) uniqueness of the Bessel functional for $\GSp(4)$:



\begin{Theorem}[\cite{NPS}, Theorem 1]\label{thm:uniqueness}
Let $\theta$ be an irreducible admissible representation of the group $G = \GSp(4)$ in a complex space $V$. Then the dimension of the space of all linear functionals $\calB$ on $V$ for which $$\calB(\theta(ut) v) = \widetilde{\psi}_A(ut)\calB(v), \textup{ for all $t \in Z_L$, $u \in \overline{U}_A$, $v \in V$}$$ does not exceed one. 
\end{Theorem}

\subsection{Existence and Uniqueness of Bessel Functionals for Principal Series Representations}\label{sec:mackey}

In this section, we will describe how we arrive at an integral realization of the Bessel functional. Explicitly, this section is dedicated to explaining how we arrive at the following result: 

\begin{Theorem}\label{thm:functional}
Let $G = \GSp(2n)$. The functional, $\calB$, on $\ind_B^G(\chi_{\textup{univ}}^{-1})$ whose restriction to functions supported on the big cell $P_A\overline{U}_A$ is given by $$\calB(\phi) = \pi^{\rho_{\ve}^{\vee}} \int_{Z_L(\goth{o})}\int_{\overline{U}_A} \psi_A(u)\phi(uz)\,du\,dz$$ is a Bessel functional.
\end{Theorem}

Essentially, we use Bruhat's extension of Mackey theory as described in \cite{Rod} to arrive at this integral realization of the functional in the rank 2 case, and then generalize. In particular, much of the argument used to prove the analogous statement for $\SO(2n+1)$ given in \cite{FG} can be applied to the rank 2 case without significant alteration, so, in the discussion to follow, we will refer the reader to the relevant results in \cite{FG} where appropriate. Before we begin, we note, per \cite{HKP}, that while the treatment in \cite{FG} ultimately yields a $\C$-valued functional on principal series representations, the method of proof applies equally well to a functional taking values in any commutative $\C$-algebra, and so the fact that $\chi_{\textup{univ}}^{-1}$ takes values in $R$ does not introduce any new complications when translating results from \cite{FG}.

Fix $G = \GSp(4)$ for the following discussion. As mentioned above, the argument that we will use to show that $\ind_B^G (\chi_{\textup{univ}}^{-1})$ admits a Bessel model, or in other words, that $$\dim \calHom_G\left(\ind_B^G(\chi_{\textup{univ}}^{-1}), \Ind_{\overline{U}_AZ_L(\goth{o})}^G \widetilde{\psi}_A\right) = 1$$
originated with Rodier in \cite{Rod}, and it makes use of the following theorem of Bruhat:

\begin{Theorem}[\cite{Bru}]\label{thm:mackey}
Let $G$ be a locally compact, totally disconnected unimodular group. Let $H_1$ and $H_2$ be two closed subgroups of $G$, and $\delta_i$ the module of $H_i$. Let $\tau_i$ be a smooth representation of $H_i$ in the vector space $E_i$, $\pi_i$ be the induced representation $\Ind_{H_i}^G \tau_i$ in the Schwarz space of $\tau_i$.

Then the space of all intertwining forms $I$ of $\pi_1$ and $\pi_2$ is isomorphic to the space of $(E_1 \otimes E_2)$-distributions $\Delta$ on $G$ such that \begin{equation} \lambda(h_1) \ast \Delta \ast \lambda(h_2^{-1}) = (\delta_1(h_1)\delta_2(h_2))^{1/2}\Delta \circ(\tau_1(h_1)\otimes \tau_2(h_2))\end{equation} where $h_i \in H_i$ and $\lambda(x)$ is the Dirac distribution in $x$. The correspondence between $I$ and $\Delta$ is given by \begin{equation}\label{eqn:intertwiningform}I(p_1(f_1), p_2(f_2)) = \int_G dg_2\int_G f_1(g_1g_2)\otimes f_2(g_2)\,d\Delta(g_1),\end{equation} where $f_i$ are locally constant functions on $E_i$ with compact support, and $p_i$ is the projection from this space of functions to the Schwarz space of $\tau_i$.
\end{Theorem}

Let $\calD(X,R)$ denote the space of $R$-distributions on a locally compact, totally disconnected space $X$. Following \cite{FG}, we begin by noting that $$\calHom_G\left(\ind_B^G \chi_{\textup{univ}}^{-1}, \Ind_{\overline{U}_AZ_L(\goth{o})}^G \widetilde{\psi}_A\right) \simeq \calHom_G\left(\ind_{\overline{U}_AZ_L(\goth{o})}^G \widetilde{\psi}_A^{\ast}, \ind_B^G (\chi_{\textup{univ}}^{-1})^{\ast}\right),$$ where $\widetilde{\psi}_A^{\ast}$ and $(\chi_{\textup{univ}}^{-1})^{\ast}$ are the smooth contragredients of $\widetilde{\psi}_A$ and $\chi_{\textup{univ}}^{-1}$, respectively. Then, by Theorem \ref{thm:mackey}, this latter space is isomorphic to the subspace $\calD_{\widetilde{\psi}_A,\chi_{\textup{univ}}^{-1}}(G,R)$ of $\calD(G,R)$ of $R$-distributions $\Delta$ on $G$ satisfying \begin{equation}\label{mackey}\lambda(b) \ast \Delta \ast \lambda(h^{-1}) = \delta_B^{1/2}(b)\chi_{\textup{univ}}^{-1}(b)\widetilde{\psi}_A^{\ast}(h)\Delta.\end{equation} for all $h \in \overline{U}_AZ_L(\goth{o})$ and $b \in B$. With this condition in mind, we will use the double-coset decomposition of $G$ suggested in the following lemma to analyze $\calD_{\widetilde{\psi}_A,\chi_{\textup{univ}}^{-1}}(G,R)$:

\begin{Lemma}\label{lemma:bruhat}
Let $g \in G=\GSp(4)$. Then $g \in Bwx_{-\alpha_1}(F)\overline{U}_AZ_L(\goth{o})$, where $w$ can be chosen from $\{1,s_2,s_1s_2,s_2s_1s_2\}$.
\end{Lemma}
\begin{proof}
Using the Bruhat decomposition, we can write $g =bwu$, where $b \in B$, $w \in W$, and $u \in \overline{U}$. Note that $s_1 \in W$ can be written as the product of a diagonal matrix, $d$, and the matrix $$\omega = \begin{pmatrix} & \omega_1 & & \\ 1 & & & \\ & & & -\omega_1 \\ & & -1 & \end{pmatrix}  \in Z_L.$$ Hence, if $w = w's_1$, where $\ell(w')< \ell(w)$ (here $\ell(w)$ denotes the length of $w$), then $g$ can be written as $bw'd\omega u = (bd')w'\omega u$, where $bd' = bw'd(w')^{-1} \in B$ and $w' \in \{1,s_2,s_1s_2, s_2s_1s_2\}$. Factoring $u = x_{-\alpha_1}(t)u_A$ for some $t \in F$ and $u_A \in \overline{U}_A$, we can see that $\omega x_{-\alpha_1}(t)u_A =  x_{\alpha_1}(t')u_A\omega$, for some $t' \in F$. Then $w'x_{\alpha_1}(t) = b'w'$ for some $b' \in B$, so that we have $g = (bdb')w'u_A\omega \in Bw'\overline{U}_AZ_L(\goth{o})$, as desired.

\end{proof}

In a series of results in \cite{FG}, starting with Proposition 2.4, Friedberg and Goldberg show that, for a given non-zero $\Delta \in \calD_{\widetilde{\psi}_A,\chi_{\text{univ}}^{-1}}(G,R)$, $\Delta$ can only be supported on one specific double coset, and that, in addition, $\Delta$ is completely determined by its restriction to that double coset. The same thing is true in our case, and we will show that the only double coset in the refined double coset \eqref{eqn:bruhat} that can serve as the support of $\Delta$ is $B\overline{U}_AZ_L(\goth{o})$. 

\begin{proof}[Proof of Theorem \ref{thm:functional}]
Following the proof of Proposition 2.4 in \cite{FG}, we will start by showing that many double cosets in Lemma \ref{lemma:bruhat} fail to satisfy the following compatibility criterion (Theorem 1.9.5~in cite \cite{Sil}): For a given double coset $Bwu_{-\alpha_1}\overline{U}_AZ_L(\goth{o})$ (where $u_{-\alpha_1} \in x_{-\alpha_1}(F)$), if there exists $b \in B$ such that $w^{-1}bw \in \overline{U}_AZ_L(\goth{o})$ and \begin{equation}\label{compatibility} \chi_{\textup{univ}}^{-1}(b) \neq \widetilde{\psi}_A(w^{-1}bw),\end{equation} then the double coset in question is not part of the support of any distribution in $\calD_{\widetilde{\psi}_A,\chi_{\textup{univ}}^{-1}}(G,R)$. 

To begin, let $u_A \in \overline{U}_A$ such that $u_{-\alpha_1}u_Au_{-\alpha_1}^{-1} \in x_{-\alpha_2}(F)$. Then, since $w(-\alpha_2) \in \Phi^+$ for $w\in \{s_2, s_1s_2, s_2s_1s_2\}$, we know that $wu_{-\alpha_1}u_Au_{-\alpha_1}^{-1}w^{-1}$ is contained in some positive root subgroup in $U$. Back in Section \ref{sec:besselmodel}, we chose $A$ such that $\omega_1 \in F^{\ast}\bs (F^{\ast})^2$; under this assumption, we can pick $u_A$ such that $\widetilde{\psi}_A(u_A)\neq 1$, as verified by some routine root subgroup calculations. Then, since $\chi_{\text{univ}}^{-1}(u) = 1$ for all $u \in U$, we see that \eqref{compatibility} does not hold on $Bwu_{-\alpha_1}\overline{U}_AZ_L(\goth{o})$ for any $w \in \{s_2, s_1s_2, s_2s_1s_2\}$ or $u_{-\alpha_1} \in x_{-\alpha_1}(F)$.


At this point, the remainder of the proof that $$\dim \calHom_G\left(\ind_B^G(\chi_{\textup{univ}}^{-1}), \Ind_{\overline{U}_AZ_L(\goth{o})}^G \widetilde{\psi}_A\right) \leq 1$$ is analogous to the end of the proof of Theorem 2.1 in \cite{FG}. We will leave the proof of the existence of a non-zero Bessel functional for Section \ref{sec:intertwiner}.

The Bessel functional is realized as an integral using Theorem \ref{thm:mackey}. In particular, Theorem \ref{thm:mackey} tells us that, if $\Delta$ is a non-zero element of $\calD_{\widetilde{\psi}_A,\chi_{\textup{univ}}^{-1}}(G,R)$, then the corresponding intertwining form, $I$, of $\ind_B^G (\chi_{\textup{univ}}^{-1})$ and $\Ind_{U_AZ_L(\goth{o})}^G \widetilde{\psi}_A$, is given by \eqref{eqn:intertwiningform}. Hence, the corresponding Bessel functional is realized as the inner integral of $I$, which in this case is 
\begin{align*}
\calB(\phi)(g) &= \int_G \phi(hg)\,d\Delta(h)
\\&= \int_{Z_L(\goth{o})}\int_{\overline{U}_A} \psi_A(u)\phi(uzg)\,du\,dz,
\end{align*}
with $g$ set equal to $1$. It is readily verified that generalizing this integral to $\GSp(2n)$ yields a Bessel functional for $\Ind_B^G(\chi_{\textup{univ}}^{-1})$, as defined in Section \ref{sec:besselmodel}. As mentioned in the previous paragraph, we will show that this integral is non-zero in the proof of Lemma \ref{besselvalue}. 
\end{proof}

\begin{remark}
Note that, in the statement of Theorem \ref{thm:functional} we have normalized the Bessel functional so that the diagram \eqref{diagram} will commute with $v_{\ve} = \pi^{\rho_{\ve}^{\vee}}$ as in Theorem \ref{thm:main}. 
\end{remark}


Letting $G = \GSp(2n)$ once more, we conclude this section with the following proposition regarding the convergence of $\calB$: 

\begin{Proposition}
If $\phi \in \ind_B^G (\chi_{\textup{univ}}^{-1})$ then $\calB(\phi)$ converges in $R$.
\end{Proposition}
\begin{proof}
Following Section 6.2 of \cite{HKP}, we begin by showing that $\calB(\phi)$ converges in a particular completion of $R$. Let $\calJ = \{-\alpha^{\vee} \mid \alpha \not\in \Phi_L^+\}$, and let $\C[\calJ]$ denote the subalgebra of $R$ generated by $\calJ$. Denote the completion of $\C[\calJ]$ with respect to the maximal ideal generated by $\calJ$ by $R_{\calJ}$. Our initial claim is that $\calB(\phi) \in R_{\calJ}$. Note that since $\phi \in \ind_B^G \chi_{\textup{univ}}^{-1}$ is compactly supported $\textup{mod}\, B$, there is no need to include any positive coroots in $\calJ$ to ensure convergence of the functional in $R_{\calJ}$. Additionally, since $\calB(\phi)$ is an integral over $\overline{U}_AZ_L(\goth{o})$, we can see that there is no need to include $\{-\alpha^{\vee} \mid \alpha \in \Phi_L^+\}$ in $\calJ$ either. Then, in order to see that $\calB(\phi)$ actually converges in $R_{\calJ}$, we apply the following lemma from \cite{HKP}:

\begin{Lemma}[\cite{HKP}, Lemma 1.10.1]
Let $\mu \in X_{\ast}(T)$. Then the set $\overline{U} \cap \pi^{\mu}UK$ is compact.
\end{Lemma}

Finally, we observe that, due to the oscillation of the character $\psi_A$, all but finitely many of the coefficients of the Laurent series $\calB(\phi)$ will vanish, which means that $\calB(\phi)$ is, in fact, an element of $R$, not just $R_{\calJ}$.
\end{proof}

\begin{remark}
Note that, if we were to specialize $\chi_{\textup{univ}}^{-1}$ to a $\C$-valued character on $B$, we could show that the resulting functional converges in $\C$ on elements of the corresponding principal series representation using an argument analogous to that presented in Section 3 of \cite{FG} (cf.~Proposition 3.5).
\end{remark}

\section{The Bessel Functional as a Hecke Algebra Intertwiner} \label{sec:intertwiner}

In this section we will prove Theorem \ref{thm:main}. This proof relies on exploiting the connection between the generators $T_s$ of $\calH_0$ and the principal series intertwining operators $A_s$. In Sections \ref{sec:psio} and \ref{sec:intercalcn} we introduce these intertwining operators and describe how they interact with $T_s$ and the Bessel functional before offering the proof of Theorem \ref{thm:main} in Section \ref{sec:proof}.

\subsection{Principal Series Intertwining Operators}\label{sec:psio}

As mentioned above, the principal series intertwining operators turn out to be closely connected to the left action of the elements of the finite Hecke algebra on $\ind_B^G(\chi_{\textup{univ}}^{-1})^J$, and we will exploit this connection in order to show that our functional acts as a Hecke algebra intertwiner in the way predicted in Theorem \ref{thm:main}. 

Our initial goal is to define a family of intertwining operators, one for each $w \in W$, that take $\M$ to itself. Our first guess at such an operator $$\calI_w : \phi \mapsto \int_{U \cap w\overline{U}w^{-1}} \phi(w^{-1}ug)\,du,$$ does not quite work, because it does not preserve $\M$. As shown in Section 1.10 of \cite{HKP}, one can extend $\M$ by scalars to a completion of $R$ according to the roots $$\Phi_{w}^+ := \{\alpha \in \Phi^+ \mid w^{-1}(\alpha) \in \Phi^-\},$$ such that this extension of $\M$ is preserved by $\calI_w$. Instead of doing this, we choose to use normalized versions of these intertwiners, $A_w$, where $$A_w := \left( \prod_{\alpha \in \Phi^+} (1-\pi^{\alpha^{\vee}}) \right)\calI_w,$$ since, using basic properties of $\calI_w$ recorded in Lemma 1.13.1 in \cite{HKP}, we can see that $A_w$ preserves $\M$. Now, since $A_w \in \End_{\calH}(\M)$, we can regard $A_w$ as an element of $\calH$ acting on the left of $\M$. In particular, for a simple reflection $s_{\alpha}$, one can show that the desired relation between $A_{s_{\alpha}}$ and $T_{s_{\alpha}}$ is \begin{equation} \label{heckeintertwiner} A_{s_{\alpha}} = (1-q^{-1})\pi^{\alpha^{\vee}} + q^{-1}(1-\pi^{\alpha^{\vee}})T_{s_{\alpha}}.\end{equation}

We pause here to note that it was Rogawski in \cite{Rog} who first used \eqref{heckeintertwiner} to recover earlier results of Rodier and others on the structure of the unramified principal series representations. However, Rogawski was using \eqref{heckeintertwiner} to recover information about the intertwining operators from his knowledge of the Hecke algebra action, which is the opposite of what we will do.

\subsection{Calculating Intertwining Factors}\label{sec:intercalcn}

In order to prove Theorem \ref{thm:main}, we will use \eqref{heckeintertwiner} to reduce the problem to understanding the interaction between the principal series intertwiners and the functional. In particular, if we make the assumption that the Bessel functional is unique, then, since $\calB \circ A_{s_{\alpha}}$ is a Bessel functional on $\ind_B^G(s_{\alpha}\cdot \chi_{\textup{univ}}^{-1})$, we know that it must be a constant multiple of $s_{\alpha} \circ \calB$. Hence, for each simple root $\alpha$, we want to calculate $c_{\alpha} \in R$ such that $$\calB \circ A_{s_{\alpha}} = c_{\alpha}(s_{\alpha} \circ \calB).$$ This turns out to be a tractable calculation, yielding the following results:



\begin{Proposition} \label{intertwiners}
Assume Theorem/Conjecture \ref{conjecture2}. With notation as above, we have that

\begin{equation} \label{intertwiner1}
\calB \circ A_{s_i} = (1-q^{-1}\pi^{\alpha_i^{\vee}})(s_{i}\circ \calB), \,\textup{if $i<n$,}
\end{equation}
and
\begin{equation} \label{intertwiner2}
\calB \circ A_{s_n} = (\pi^{\alpha_n^{\vee}}-q^{-1})(s_n\circ \calB).
\end{equation}
\end{Proposition}

\begin{remark}
Note that we need Theorem/Conjecture \ref{conjecture2} in order to assert that $\calB \circ A_{s_{\alpha}}$ is a scalar multiple of $s_{\alpha} \circ \calB$ in the rank $n>2$ case. After this point, the rest of the proof of Theorem \ref{thm:main} procedes with no caveats.
\end{remark}





In order to prove \eqref{intertwiner1}, we will need to calculate the image of the Iwahori-fixed vectors $\phi_1$ and $\phi_{s_i}$, for $i<n$, in the model: 

\begin{Lemma}\label{besselvalue}
The Bessel functional takes on the following values on the following Iwahori-fixed vectors:
\begin{equation} \label{phi1}
 \calB(\phi_1) = \pi^{\rho_{\ve}^{\vee}} m(\overline{U}_A Z_L(\goth{o})\cap BJ),
\end{equation}
and
\begin{equation} \label{phis}
\calB(\phi_{s_i}) = \pi^{\rho_{\ve}^{\vee}} m(\overline{U}_AZ_L(\goth{o}) \cap Bs_iJ), \text{ if $i < n$}.
\end{equation}
Moreover, these values are non-zero, as the sets $\overline{U}_AZ_L(\goth{o}) \cap BJ$ and $\overline{U}_AZ_L(\goth{o}) \cap Bs_iJ$ have non-zero measure.
\end{Lemma}

\begin{remark}
We already know that the integrals $\calB(\phi_1)$ and $\calB(\phi_{s_i})$ converge in $R$ from Section \ref{sec:mackey}; however, it will be important to the proof of Proposition \ref{intertwiners} for us to show that they are non-zero and invariant under the composition $s_i \circ \calB$.
\end{remark}

Before we can prove Lemma \ref{besselvalue}, we must first prove the following lemma:

\begin{Lemma}\label{blockiwahori}
$\overline{U}_A \cap P_AJ_A = \overline{U}_A \cap J_A$.
\end{Lemma}
\begin{proof}
Let $u \in \overline{U}_A \cap P_AJ_A$. We see that the standard argument for the rank 1 Iwahori factorization $J = (J\cap B)(J\cap \overline{U})$ can be adapted here to give $J_A = (J_A \cap P_A)(J_A \cap \overline{U}_A)$. Using this, we see that we can factor $u = pj$, with $p \in P_A$ and $j \in J_A \cap \overline{U}_A$. Rewriting this as $uj^{-1} = p$, we see that $uj^{-1} \in \overline{U}_A \cap P_A$, so $u = j \in J_A\cap \overline{U}_A$.
\end{proof}

\begin{proof}[Proof of Lemma \ref{besselvalue}]
Consider the Iwahori-Bruhat-like decomposition $$G = P_AJ_A \sqcup P_As_nJ_A,$$ where $J_A$ is the preimage of $P_A(k)$ under the canonical homomorphism $G(\goth{o}) \to G(k)$ (note that $P_A$ is the parabolic subgroup generated by $B$ and the root subgroups $x_{-\alpha_i}(F)$ for $i < n$). In order to see that \begin{equation} \label{eqn3} \calB(\phi_1) = \int_{Z_L(\goth{o})} \int_{\overline{U}_A} \psi_A(u) \phi_1(uz) \,du\,dz = m(\overline{U}_AZ_L(\goth{o}) \cap BJ),\end{equation} we must first show that \begin{equation}\label{eqn1} \overline{U}_AZ_L(\goth{o}) \cap BJ \subset (\overline{U}_A \cap J_A)J_A. \end{equation} 

Now, if $u \in \overline{U}_A$ and $z \in Z_L(\goth{o})$, then $uz \in BJ$ only if $u$ has an Iwahori-Bruhat decomposition $u = bwj$ with $b \in B$, $j \in J$, and $w \in W_L$, since \begin{equation} Z_L(\goth{o}) \subset J_A = \bigsqcup_{w\in W_L} JwJ .\end{equation} Additionally, we see that $\overline{U}_A \cap BwJ \subset \overline{U}_A\cap P_AJ_A$ whenever $w \in W_L$, so that we have \begin{equation} \label{eqn2} \overline{U}_AZ_L(\goth{o}) \cap BJ \subset (\overline{U}_A \cap P_AJ_A)J_A. \end{equation} Equation \eqref{eqn1} now follows from \eqref{eqn2} by Lemma \ref{blockiwahori}.

Since the conductor of $\psi_A$ is $\goth{o}$, \eqref{eqn1} tells us that \begin{equation} \int_{Z_L(\goth{o})} \int_{\overline{U}_A} \psi_A(u) \phi_1(uz) \,du\,dz = \int_{Z_L(\goth{o})} \int_{\overline{U}_A} \phi_1(uz) \,du\,dz,\end{equation} and so we see that \eqref{eqn3} holds. Finally, we note that $$(\overline{U}_A \cap J)(Z_L(\goth{o}) \cap J) \subset \overline{U}_AZ_L(\goth{o}) \cap BJ,$$ which means that $\calB(\phi_1) \neq 0$.

Making suitable adjustments to the argument given above gives us \eqref{phis}.
\end{proof}

\begin{remark}
The proof of Theorem \ref{thm:functional} is now complete as well.
\end{remark}

We make the choice now to normalize our Haar measure so that $m(\overline{U}_AZ_L(\goth{o}) \cap BJ) = 1$. We are ready to prove Proposition \eqref{intertwiners}:

\begin{proof}[Proof of Proposition \ref{intertwiners}]
In order to make our calculation of $c_{\alpha_i}$ easier, for $i<n$, we will evaluate $\calB \circ A_{s_i}$ on the Iwahori-fixed vector $\phi_1+\phi_{s_i}$. From Lemma 1.13.1 in \cite{HKP}, we know that $$\calB(A_{s_i}(\phi_1 + \phi_{s_i})) = (1-q^{-1}\pi^{\alpha_i^{\vee}})\calB(\phi_1 + \phi_{s_i}).$$ Note that, if we can show that $\calB(\phi_1)$ and $\calB(\phi_{s_i})$ are both invariant under the reflection $s_i$, then we will have proved \eqref{intertwiner1}. By Lemma \ref{besselvalue}, we know that $\calB(\phi_1) = \pi^{\rho_{\ve}^{\vee}},$ and hence $\calB(\phi_1)$ is invariant under the reflection $s_i$. Similarly, from \eqref{phis}, we know that $\calB(\phi_{s_i})$ is a non-zero multiple of $\pi^{\rho_{\ve}^{\vee}}$, and so we see that $\calB(\phi_{s_i})$ is also invariant under the reflection $s_i$.



Next, we calculate $c_{\alpha_n}$. Finding this intertwining constant is similar to the corresponding calculation for the Whittaker functional on $\GL(2)$. Let $\phi$ be an element of $\ind_B^G(\chi_{\textup{univ}}^{-1})^J$ on which $\calB$ is non-zero. A priori, we do not know that such an element exists - however, in our proof of \eqref{intertwiner1} we showed that $\phi_{1}$ is such a function. We see that $$\calB(A_{s_{n}}\phi)(1) = \pi^{\rho_{\ve}^{\vee}} \int_{Z_L(\goth{o})} \int_{\overline{U}_A} \int_{F}\psi_A(u) \phi(s_nx_{\alpha_n}(\tau)uz)\,d\tau\,du\,dz;$$ note that we only need to evaluate the functional at 1 in order to determine the intertwining constant. Using the rank 1 Bruhat decomposition $$s_nx_{\alpha_n}(\tau) = h_{\alpha_n}(\tau^{-1})x_{\alpha_n}(\tau) x_{-\alpha_n}(\tau^{-1}),$$ where $h_{\alpha_n}$ denotes the semisimple subgroup of the embedded $\SL(2)$ triple corresponding to $\alpha_n$, and excluding the point $\tau = 0$, we can rewrite this integral as $$\int_{Z_L(\goth{o})}\int_{\overline{U}_A} \int_{F^{\times}}\psi_A(u)\chi_{\textup{univ}}^{-1}(h_{\alpha_n}(\tau^{-1}))\phi(x_{-\alpha_n}(\tau^{-1})uz)\,d\tau\,du\,dz.$$ After factoring $u$ into root subgroups and performing a linear change of variables, we find that

\begin{align*}
\calB(A_{s_{\alpha_n}}\phi)(1) &= \pi^{\rho_{\ve}^{\vee}} \int_{Z_L(\goth{o})} \int_{\overline{U}_A} \psi_A(u)\phi(uz) \int_{F^{\times}}\psi_A(-\tau^{-1})\chi_{\textup{univ}}^{-1}(h_{\alpha_n}(\tau^{-1}))\,d\tau\,du\,dz
\\&= c_{\alpha_n}(s_{\alpha_n}\circ \calB(\phi))(1),
\end{align*}



where $$c_{\alpha_n} = \int_{F^{\times}}\psi_A(-\tau^{-1})\chi_{\textup{univ}}^{-1}(h_{\alpha_n}(\tau^{-1}))\,d\tau.$$ This last integral can be evaluated by shells so that, after normalizing the Haar measure so that $m(x_{\alpha_n}(\goth{o}))=1$, we get the familiar Whittaker intertwining constant $$c_{\alpha_n} = (\pi^{\alpha_n^{\vee}}-q^{-1}).$$



\end{proof}

\begin{remark}
Note that we were able to verify the long root intertwining constant, \eqref{intertwiner2}, on an arbitrary Iwahori-fixed vector without invoking the uniqueness of the model. Thus, in the proof of Theorem \ref{thm:main}, we only make use of Theorem/Conjecture \ref{conjecture2} when we prove \eqref{intertwiner1} (in the rank $n>2$ case).
\end{remark}

\subsection{Proof of Theorem \ref{thm:main}} \label{sec:proof}

In order to show that $\calB$ is an $\calH$-intertwiner as claimed in Theorem \ref{thm:main}, we will need to know the action of $T_{s_{\alpha}}$ on $V_{\ve} \simeq R$ explicitly for simple reflections $s_{\alpha}$. The calculation of this action follows easily from the Bernstein relation \eqref{bernstein}: for a basis element $\pi^{\mu}v_{\ve}$ - where, as before, $v_{\ve}$ denotes the eigenvector of $\calH_0$ corresponding to $\ve$ - we see that

\begin{align*}
T_{s_{\alpha}} \cdot \pi^{\mu}v_{\ve} &= \pi^{s_{\alpha}(\mu)} \ve(T_{s_{\alpha}})v_{\ve} + (1-q)\frac{\pi^{s_{\alpha}(\mu)}-\pi^{\mu}}{1-\pi^{-\alpha^{\vee}}}v_{\ve}
\\&= \left(\ve(T_{s_{\alpha}}) + \frac{1-q}{1-\pi^{-\alpha^{\vee}}}\right)\pi^{s_{\alpha}(\mu)}v_{\ve} + \frac{q-1}{1-\pi^{-\alpha^{\vee}}}\pi^{\mu}v_{\ve};
\end{align*}
in the second equality, we have rearranged terms so that we can see how $T_{s_{\alpha}}\cdot \pi^{\mu}v_{\ve}$ is expressed as a linear combination of $\pi^{\mu}v_{\ve}$ and $\pi^{s_{\alpha}(\mu)}v_{\ve}$ over $R$. Thus, regarding $T_{s_{\alpha}}$ as an operator on $R$, we see that $T_{s_{\alpha}}$ acts on $f \in R$ by \begin{equation}\label{heckeoperator} T_{s_{\alpha}}: f \mapsto \left(\ve(T_{s_{\alpha}}) + \frac{1-q}{1-\pi^{-\alpha^{\vee}}}\right)f^{s_{\alpha}} + \frac{q-1}{1-\pi^{-\alpha^{\vee}}}f.\end{equation}

\begin{proof}[Proof of Theorem \ref{thm:main}]
The main result we need to prove is that $\calB$ is indeed a left $\calH$-module intertwiner from $\ind_B^G(\chi_{\textup{univ}}^{-1})^J$ to $V_{\ve}$, where $\ve$ is the character that acts by multiplication by $-1$ on long simple roots and acts by $q$ on short simple roots. Once we have done this and checked that $\calF(1_{T(\goth{o})UJ}) = \calB(\phi_1) = \pi^{\rho_{\ve}^{\vee}}$, we can see that the diagram commutes since $\ind_B^G(\chi_{\textup{univ}}^{-1})^J \simeq M \simeq \calH$. That the diagram commutes on $1_{T(\goth{o})UJ}$ is immediate - we know that $\calB(\phi_1) = \pi^{\rho_{\ve}^{\vee}}$ from Lemma \ref{besselvalue}, and we observe that $\calF(1_{T(\goth{o})UJ}) = \calF(1_{T(\goth{o})UJ} \ast 1_J) = \pi^{\rho_{\ve}^{\vee}}$.


In order to prove that $\calB$ is a left $\calH$-module intertwiner, it suffices to show, on a set of generators $\{h\}$ for $\calH$, that $$\calB(h\cdot \phi) = h\cdot \calB(\phi), \textup{ for any $\phi \in \ind_B^G(\chi_{\textup{univ}}^{-1})^J.$}$$ In particular, we will choose our set of generators to be those elements of the form $\pi^{\mu}T_{s_{\alpha}}$ where $\mu \in X_{\ast}(T)$ and $s_{\alpha}$ is a simple reflection. Since $\pi^{\mu}$ acts by translation on both $V_{\ve}$ and $\ind_B^G(\chi_{\textup{univ}}^{-1})^J$, we can reduce to checking the equality on $T_{s_{\alpha}}$.

From \eqref{heckeintertwiner}, we immediately see that $$q^{-1}(1-\pi^{\alpha^{\vee}})\calB(T_{s_{\alpha}}\cdot \phi) = \calB(A_{s_{\alpha}}\phi)- (1-q^{-1})\pi^{\alpha^{\vee}}\calB(\phi).$$ Applying Proposition \ref{intertwiners}, we see that $$q^{-1}(1-\pi^{\alpha^{\vee}})\calB(T_{s_{\alpha}}\cdot \phi) = \left\{\begin{array}{ll} (1-q^{-1}\pi^{\alpha^{\vee}})(s_{\alpha}\circ \calB)(\phi) + (q^{-1}-1)\pi^{\alpha^{\vee}}\calB(\phi) & \textup{if $\alpha = \alpha_i$ $(i<n)$} \\ (\pi^{\alpha^{\vee}}-q^{-1})(s_{\alpha}\circ \calB)(\phi) + (q^{-1}-1)\pi^{\alpha^{\vee}}\calB(\phi) & \textup{if $\alpha = \alpha_n$.} \end{array}\right.$$ Dividing by $q^{-1}(1-\pi^{\alpha^{\vee}})$, we see that the operator acting on $B(\phi)$ is $$f \mapsto \frac{q}{1-\pi^{\alpha^{\vee}}}\left\{\begin{array}{ll} (1-q^{-1}\pi^{\alpha^{\vee}}) f^{s_{\alpha}} + (q^{-1}-1)\pi^{\alpha^{\vee}}f & \textup{if $\alpha = \alpha_i$ $(i<n)$} \\ (\pi^{\alpha^{\vee}}-q^{-1})f^{s_{\alpha}} + (q^{-1}-1)\pi^{\alpha^{\vee}}f & \textup{if $\alpha = \alpha_n$.} \end{array}\right.$$


If we compare this with the operator in \eqref{heckeoperator} that described the action of $T_{s_{\alpha}}$ on $R$, we see that it matches it exactly in both cases (recall that $\ve(T_{s_i}) = q$, if $i<n$, and $\ve(T_{s_n}) = -1$). Thus, $\calB(T_{s_{\alpha}}\cdot \phi) = T_{s_{\alpha}}\cdot \calB(\phi)$ for any $\phi \in \ind_B^G(\chi_{\textup{univ}}^{-1})$ and simple reflection $s_{\alpha}$.
\end{proof}

\section{Calculating Distinguished Vectors at Torus Elements} \label{sec:spherical}

In this section, we will focus on calculating the images of distinguished vectors in unique models of the universal principal series of $\GSp(2n)$. In \ref{sec:besspherical}, we will conclude our discussion of the Bessel functional with a proof of Theorem \ref{thm:spherical}. Then, in \ref{sec:womodels}, we will move on to discussing the connection between the Whittaker-Orthogonal models defined in \cite{BFG} and the proposed $\calH$-intertwiner corresponding to the fourth character, $\sigma$, of the finite Hecke algebra of $\GSp(4)$.

\subsection{Calculating Distinguished Vectors in the Bessel Model} \label{sec:besspherical}

We will use Theorem \ref{thm:main} to calculate the images of certain distinguished vectors under $\calB$ on anti-dominant, integral torus elements, culminating with a proof of Theorem \ref{thm:spherical}. Before we can prove Theorem \ref{thm:spherical}, we will calculate the images of the Iwahori-fixed vectors $\phi_w = \eta(1_{T(\goth{o})UwJ})$ in terms of the action of $\calH$ on $v_{\ve} = \pi^{\rho_{\ve}^{\vee}}$, as described in Theorem \ref{thm:iwahori}. Using the linearity of $\calB$ along with the alternator formula developed in \cite{BBF2}, we will arrive at a proof of Theorem \ref{thm:spherical}, which we will show matches the expression obtained in Corollary 1.8 in \cite{BFF}. 

In order to prove Theorem \ref{thm:iwahori}, we must first prove the following Iwahori factorization:

\begin{Proposition} \label{prop:iwahori}
$J = (J\cap B)(J\cap \overline{U}_AZ_L(\goth{o}))$.
\end{Proposition}

The proof of this proposition relies on the same result in the rank 2 case, which we prove now as a separate lemma:

\begin{Lemma} \label{iwahoria}
If $G=\GSp(4)$, then $J = (J\cap B)(J\cap \overline{U}_AZ_L(\goth{o}))$.
\end{Lemma}
\begin{proof}
Using the usual Iwahori factorization, we can see that it suffices to show that the subgroup $x_{-\alpha_1}((\pi))$ of $J$ is contained in $(J \cap B)(J \cap \overline{U}_AZ_L(\goth{o}))$. To see that this is the case, observe that, for $\tau = u\pi^j$ with $u \in \goth{o}^{\times}$ and $j > 0$, we can factor $x_{-\alpha_1}(\tau) = bh$, where $$b = \begin{pmatrix} g & 0 \\ 0 & \det g \cdot (g')^{-1} \end{pmatrix}\textup{ with } g = \begin{pmatrix} (1-\omega_1 \tau^2)^{-1} & -\omega_1 \tau(1-\omega_1 \tau^2)^{-1}\\ 0 & 1\end{pmatrix},$$ and $$h = \begin{pmatrix} \gamma & 0 \\ 0 & \det \gamma\cdot (\gamma')^{-1} \end{pmatrix} \textup{ with } \gamma = \begin{pmatrix} 1 & \omega_1 \tau \\ \tau & 1 \end{pmatrix}.\footnote{Recall that the parameter $-\omega_1$ was defined to be an entry of $A$ back in \ref{sec:besselmodel}.}$$ 
\end{proof}

\begin{proof}[proof of Proposition \ref{prop:iwahori}]
We begin by noting that, using the usual Iwahori factorization (as in Lemma \ref{iwahoria}), it suffices to show that every element in $J \cap \overline{U} \cap L_A$ is contained in $(J \cap B)(J\cap \overline{U}_AZ_L(\goth{o}))$. Let $\overline{u} \in J \cap \overline{U} \cap L_A$. We can factor $\overline{u}$ into a product of elements from the root subgroups contained in $L_A$, so that $$\overline{u} = \prod_{\alpha \in \Phi_L^+} u_{-\alpha},$$ where $u_{-\alpha} \in x_{-\alpha}((\pi))$. Note that, if $\overline{u}$ has no nontrivial factors when it is factored into root subgroups, then $\overline{u} = 1$. Now, suppose that $\overline{u}$ has $k$ distinct nontrivial factors when it is factored into root subgroups as $\overline{u} = \prod_{\alpha \in \Phi_L^+} u_{-\alpha}$. Observe that, since each of the roots in $\Phi_L$ is a short root, if $\alpha \in \Phi_L^+$, then $\alpha = \sum_{i=1}^{n-1} c_i\alpha_i$ where $c_i \in \{0,1\}$; let $c(\alpha) := \sum_{i=1}^{n-1} c_i$. Write $\overline{u}$ as a product where the $u_{-\alpha}$'s are ordered from left to right by increasing $c(\alpha)$. Let $\beta \in \Phi_L^+$ be the root such that $u_{-\beta}$ is the rightmost factor in $\overline{u}$ as described above. Then, by Lemma \ref{iwahoria}, we can write $u_{-\beta} = t_{\beta}u_{\beta}z_{\beta}$, with $t_{\beta} \in T\cap J$, $u_{\beta} \in x_{\beta}(\goth{o})$, and $z_{\beta} \in J \cap \overline{U}_AZ_L(\goth{o})$. Observe that $t_{\beta}^{-1}u_{-\alpha}t_{\beta} \in x_{-\alpha}((\pi))$, so that moving $t_{\beta}$ all the way to the left in this factorization of $\overline{u}$ leaves us with a factorization of $t_{\beta}^{-1}\overline{u}(u_{\beta}z_{\beta})^{-1}$ into elements from the same root subgroups, minus $x_{\beta}$, in the same order that they were in in the initial factorization of $\overline{u}$. Let $u_{-\alpha}$ now refer to the element of $x_{-\alpha}((\pi))$ in the factorization of $t_{\beta}^{-1}\overline{u}(u_{\beta}z_{\beta})^{-1}$ into root subgroups.




We would like to show that we can move $u_{\beta}$ across $\prod_{\alpha \neq \beta} u_{-\alpha}$ and end up with $b\left(\prod_{\alpha \neq \beta} u_{-\alpha}\right)z_{\beta}$. To this end, we will make use of the following properties (assume $\alpha_1,\alpha_2 \in \Phi_L^+$):

\begin{enumerate}
\item \label{fact1} If $c(\alpha)< c(\alpha')$ then $x_{-\alpha}(t)x_{\alpha'}(s) = x_{\alpha'}(s)x_{\alpha'-\alpha}(ts)x_{-\alpha}(t)$, and if $\alpha'-\alpha \in \Phi$ then $\alpha'-\alpha \in \Phi_L^+$ and $c(\alpha'-\alpha) < c(\alpha')$; otherwise $x_{-\alpha}(t)$ and $x_{\alpha'}(s)$ commute.
\item If $c(\alpha)> c(\alpha')$ then $x_{-\alpha}(t)x_{\alpha'}(s) = x_{\alpha'}(s)x_{\alpha'-\alpha}(ts)x_{-\alpha}(t)$, and if $\alpha'-\alpha \in \Phi$ then $\alpha'-\alpha \in \Phi_L^-$ and $c(\alpha'-\alpha) < c(\alpha)$; otherwise $x_{-\alpha}(t)$ and $x_{\alpha'}(s)$ commute.
\item If $c(\alpha) = c(\alpha')$ but $\alpha'\neq \alpha$, then $x_{-\alpha}(t)$ and $x_{\alpha'}(s)$ commute.
\item Thinking of $x_{\alpha}(t)$ as a subgroup of $\GL(2)$ embedded in $G$, $$x_{-\alpha}(t)x_{\alpha}(s) = \begin{pmatrix} (1+ts)^{-1} & \\  & 1+ts\end{pmatrix} \begin{pmatrix} 1 & s(1+ts) \\ & 1 \end{pmatrix} \begin{pmatrix} 1 &  \\ t(1+ts)^{-1} & 1 \end{pmatrix}.$$
\end{enumerate}



From these properties, we see that when we move $u_{\beta}$ across $\prod u_{-\alpha}$ and next to $t_{\beta}$, we are left with an element that we can factor into root subgroups where the roots in question may be in $\Phi_L^+$ or $\Phi_L^-$, but we know that for each such $\alpha$, $c(\alpha) < c(\beta)$ (or $c(-\alpha) < c(\beta)$). At this point, we move each factor of $x_{\alpha}(t)$ with $\alpha \in \Phi_L^+$ left until there are no factors of the form $x_{\alpha'}(s)$ with $\alpha' \in \Phi_L^-$ to its left, starting with the leftmost such factor. 

Additionally, we observe from these properties that commuting $x_{\alpha}(s)$ past $x_{\alpha'}(t)$ (with $\alpha,\alpha'$ as in the previous sentence) will produce either (a) a factor of $x_{\alpha+\alpha'}(st)$ where either $\alpha+\alpha' \in \Phi_L^-$ or $\alpha+\alpha' \in \Phi_L^+$ with $c(\alpha+\alpha')< c(\alpha)$, (b) a factor $t_{\alpha}$ in $h_{\alpha}(\goth{o})$ and a factor of $x_{\alpha}(s(1+ts))$, or (c) nothing, if the two factors commute. Note that, in case (b), we can move $t_{\alpha}$ left across all factors of the form $x_{\alpha''}(z)$ with $\alpha'' \in \Phi_L^-$ without creating any additional factors of root subgroups, as described earlier. Thus, we see that the process of moving factors of the form $x_{\alpha}(s)$ with $\alpha \in \Phi_L^+$ will ultimately terminate, leaving us with a factorization of $\overline{u}$ as $$\overline{u} = b \left(\prod_{\alpha \in \Phi_L^+} x_{-\alpha}(r_{\alpha})\right) z_{\beta},$$ where $b \in J \cap B$ and where we know that there is no nontrivial factor of $x_{-\beta}((\pi))$ in this product. We repeat this process on $\prod_{\alpha \in \Phi_L^+} x_{-\alpha}(r_{\alpha})$, and then again until we are left with a factorization $\overline{u} = bz$ with $b \in J \cap B$ and $z \in J \cap \overline{U}_AZ_L(\goth{o})$ - note that we can be sure that this process will terminate since (a) at each step, we are removing the representative from a specific root subgroup from the product; (b) the maximum value of $c(\alpha)$ in a given factorization is no larger than the maximum value in the previous factorization; (c) the value of $c(\alpha)$ for each new factor $u_{-\alpha}$ introduced in the process of moving $u_{-\beta}$ across the product is strictly smaller than $c(\beta)$; and (d) $|\Phi_L^+|< \infty$.



\end{proof}

\begin{proof}[proof of Theorem \ref{thm:iwahori}]
We begin by looking at the right-hand side, $T_w\pi^{\lambda} \cdot v_{\ve}$. We will use the commutativity of the diagram \eqref{diagram} and the dominance of $\lambda$ to show that $$\calB(\phi_w \ast 1_{J\pi^{\lambda}J}) = T_w\pi^{\lambda} \cdot v_{\ve},$$ so that it suffices to show that \begin{equation}\label{eqn:torus} \calB(\pi^{-\lambda}\cdot \phi_w) = \frac{1}{m(J\pi^{\lambda}J)}\calB(\phi_w \ast 1_{J\pi^{\lambda}J}).\end{equation} In order to see that this second equality holds, first note that $\pi^{-\lambda}\cdot \phi_w = \eta(1_{T(\goth{o})UwJ\pi^{\lambda}})$ by definition (here we emphasize the definition of $\eta$ as a vector-space isomorphism from $C_c(T(\goth{o})U \bs G)$ to $\ind_B^G(\chi_{\textup{univ}}^{-1})$). Now, if we look at $\calB(\phi_w \ast 1_{J\pi^{\lambda}J}) = \calB(\eta(1_{T(\goth{o})UwJ} \ast 1_{J\pi^{\lambda}J}))$, we see from the definition of the convolution that $$\calB(\phi_w \ast 1_{J\pi^{\lambda}J}) = \int_{\overline{U}_AZ_L(\goth{o})} \int_{J \bs J\pi^{-\lambda}J} \int_J \widetilde{\psi}(h)\eta(1_{T(\goth{o})UwJ})(hj\gamma)\,dj\,d\gamma\,dh.$$ Using Proposition \ref{prop:iwahori} and making the change of variables $h \mapsto hj^{-1}$, the integral above simplifies to $$\calB(\phi_w \ast 1_{J\pi^{\lambda}J}) = m(J\pi^{\lambda}J) \int_{\overline{U}_AZ_L(\goth{o})} \widetilde{\psi}(h)\eta(1_{T(\goth{o})UwJ\pi^{\lambda}})(h)\,dh,$$ since the conductor of $\psi$ is $\goth{o}$. Thus, we have established \eqref{eqn:torus}.


To see that $\calB(\phi_w \ast 1_{J\pi^{\lambda}J}) = T_w\pi^{\lambda}\cdot v_{\ve}$, we first note that $$\phi_w \ast 1_{J\pi^{\lambda}J} = \eta((1_{T(\goth{o})UJ} \ast T_w)\ast 1_{J\pi^{\lambda}J}) = \eta((T_w\pi^{\lambda})\cdot 1_{T(\goth{o})UJ}),$$ where the second equality follows because $\lambda$ is dominant. Thus, by Theorem \ref{thm:main}, we see that $$\calB(\phi_w \ast 1_{J\pi^{\lambda}J}) = T_w\pi^{\lambda} \cdot v_{\ve}.$$

\end{proof}

As noted at the beginning of the section, the linearity of $\calB$ gives us the following immediate corollary regarding $\phi^{\circ}$:

\begin{Corollary}\label{spherical}
For dominant $\lambda$, $$\calB(\pi^{-\lambda} \cdot \phi^{\circ}) = \frac{1}{m(J\pi^{\lambda}J)} \sum_{w \in W} T_w\pi^{\lambda}\cdot v_{\ve}.$$
\end{Corollary}

In order to prove Theorem \ref{thm:spherical}, we will need to make use of an identity of operators on $\Frac(R)$. Recall that when we recorded the action of $T_{s_{\alpha}}$ as an operator on $R$ in \eqref{heckeoperator}, it was with $R$ regarded as a left $\calH$-module with eigenvector 1. Our goal is to calculate the image of the spherical function in the model $V_{\ve}$, and, as noted previously, $R$ is isomorphic to $V_{\ve}$ under the isomorphism $1 \mapsto \pi^{\rho_{\ve}^{\vee}}$. Then, extending the action of $\calH$ to $\Frac(R)$, we realize the operator associated to $T_{s_{\alpha}}$ via this isomorphism (now regarded as an isomorphism of $\Frac(R)$) as $$\goth{T}_{s_{\alpha}}:= \pi^{\rho_{\ve}^{\vee}}T_{s_{\alpha}}\pi^{-\rho_{\ve}^{\vee}},$$ so that we can rewrite Corollary \eqref{spherical} as $$\calB(\pi^{-\lambda}\cdot \phi^{\circ}) = \frac{\pi^{-\rho_{\ve}^{\vee}}}{m(J\pi^{\lambda}J)}\sum_{w \in W} \goth{T}_w\pi^{\lambda+2\rho_{\ve}^{\vee}}.$$ Explicitly, the action of $\goth{T}_{s_{\alpha}}$ on $\Frac(R)$ for a simple root $\alpha$ is given by 

$$\goth{T}_{s_{\alpha}}: f \mapsto \frac{q}{1-\pi^{\alpha^{\vee}}}\left\{\begin{array}{ll} (\pi^{\alpha^{\vee}}-q^{-1})\pi^{\alpha^{\vee}} f^{s_{\alpha}} + (q^{-1}-1)\pi^{\alpha^{\vee}}f & \textup{if $\alpha = \alpha_i$ $(i<n)$} \\ (1-q^{-1}\pi^{\alpha^{\vee}})f^{s_{\alpha}} + (q^{-1}-1)\pi^{\alpha^{\vee}}f & \textup{if $\alpha = \alpha_n$.} \end{array}\right.$$

The operator identity that we will use is a deformed version of the Weyl character formula, established in \cite{BBF2} in a more general setting where $G$ is only assumed to be split, connected and reductive. Let $\Omega$ denote the operator on $\Frac(R)$ given by the Weyl character formula-like expression $$\Omega:= \pi^{-\rho^{\vee}}\prod_{\alpha \in \Phi^+} (1-\pi^{-\alpha^{\vee}})^{-1}\A(\pi^{-\rho^{\vee}}).$$ The deformation depends on the choice of character of the Hecke algebra, as described in the following theorem:

\begin{Theorem}[\cite{BBF2}, Theorem 13]\label{thm:alternator}
If we have a character $\tau$ of $\calH_0$ for $G$, then $$\sum_{w\in W} \goth{T}_w = \left(\prod_{\alpha \in \Phi_{-1}^+}(1-q\pi^{\alpha^{\vee}})\right) \Omega \left(\prod_{\alpha \in \Phi_{q}^+}(1-q\pi^{\alpha^{\vee}})\right),$$ where $\Phi_{-1}^+$, resp.~$\Phi_{q}^+$, are those positive roots that are the same length as the simple roots $\alpha$ such that $\tau(\alpha) = -1$, resp.~$\tau(\alpha)=q$.
\end{Theorem}

In the case of $\ve$, the set $\Phi_{-1}^{+}$ consists of the long positive roots, and $\Phi_{q}^{+}$ consists of the short positive roots, which leads us to Theorem \ref{thm:spherical}. The image of the spherical function in the Bessel model evaluated on torus elements was previously calculated in the case when $n=2$ by Bump, Friedberg and Furusawa in Corollary 1.8 in \cite{BFF}, and, indeed, it can be confirmed by observation that our formula matches theirs, up to normalization.\footnote{In \cite{BFF}, they work with a choice of unramified principal series $\ind_B^G \chi$ instead of the universal principal series. The parameters denoted $\alpha_1,\alpha_2$ in \cite{BBF2} can be expressed as $\alpha_1^{2} = \chi(\pi^{-(\alpha_1+\alpha_2)^{\vee}})$ and $\alpha_2^2 = \chi(\pi^{\alpha_1^{\vee}})$.}

 
\subsection{Whittaker-Orthogonal Models and the Shalika character}\label{sec:womodels}

In this section, we consider the character $\sigma$ of $\calH_0$, which was defined in Section \ref{sec:intro} to be the character which acts by $q$ on long simple roots and $-1$ on short simple roots. For each of the other three characters of $\calH_0$ on $\GSp(4)$, we have found a subgroup $S \subset G$ such that the model formed by inducing $S$ to $G$ contains that character with multiplicity one - $\sigma$ is the only character for which we have not found an explicit integral realization of $\calL$ as in the diagram \eqref{diagram}. However, even without this information, we can still say what the image of the spherical function under $\calL$ in $V_{\sigma}$, evaluated on torus elements, would have to be, by the commutativity of \eqref{diagram} combined with Theorem \ref{thm:alternator}. As stated in Proposition \ref{prop:shalika}, we will show that the result matches the image of the spherical function in the Whittaker-Orthogonal model (WO-model) defined by Bump, Friedberg, and Ginzburg in \cite{BFG}.

The WO-model is defined as a representation on $\SO(2n+2)$. Let $\overline{U}$ be the opposite unipotent radical of the parabolic subgroup of $\SO(2n+2)$ with a Levi component that is diagonal except for a central $\SO(4)$ block and let $\psi$ be a character of $\overline{U}$ defined as $$\psi(\overline{u}) := \psi_0(\overline{u}_{21} + \overline{u}_{32} + \cdots + \overline{u}_{n-1,n-2} + \overline{u}_{n+1,n-1} + \overline{u}_{n+2,n-1}),$$ where $\psi_0$ is a nontrivial additive character of $F$ with conductor $\goth{o}$. Let $Z(\psi) \simeq \SO(3)$ be the stabilizer of this character contained in the Levi. Then, for an irreducible admissible representation $\theta$ of $\SO(2n+2)$, we say that $\theta$ has a WO-model if there exists a nonzero linear functional $\WO$ on the representation space $V_{\theta}$ of $\theta$ such that $$\WO(\theta(\overline{u}h)v) = \psi(\overline{u})\WO(v),$$ for $\overline{u} \in \overline{U}_{\pi}$, $h \in Z(\psi)$, and $v \in V_{\theta}$. The uniqueness of WO-models is established in Theorem 4.1 in \cite{BFG}. The authors then show, in Theorem 4.2, that, if $\hat{\chi} = \ind_B^G(\chi)$ is an irreducible unramified principal series representation, then $\hat{\chi}$ admits a WO-model if and only if $\hat{\chi}$ is a \emph{local lifting} of an unramified principal series representation of $\Sp(2n)$. We call $\hat{\chi}$ a local lifting from $\Sp(2n)$ if one of the Langlands' parameters of $\hat{\chi}$ is 1. The authors note that this is in conformity with Langlands' functoriality since the L-group of $\Sp(2n)$ is $\SO(2n+1)$, and that if one of the Langlands' parameters is 1 then the given conjugacy class is in the image of the inclusion of L-groups $\SO(2n+1) \into \SO(2n+2)$.\footnote{Recall that the L-group of $\SO(2n+2)$ is $\SO(2n+2)$.}
Now, suppose that we have an unramified principal series representation of $\SO(6)$, $\hat{\chi} = \ind_B^G \chi$, where $\chi = (\chi_1,\hdots,\chi_{n+1})$ with $\chi_1,\hdots,\chi_{n+1}$ quasicharacters of $F^{\times}$. Let $z_i = \chi_i(\pi)$ for each $i$, and let $z = \chi(\pi)$. Then, if $z_{n+1} = 1$, we have the following formula from \cite{BFG} for the image of the spherical vector of $\hat{\chi}$ under $\WO$ evaluated at torus elements $\pi^{\lambda}$, with $\lambda = (\lambda_1,\lambda_2,0)$: 

\begin{Theorem}[\cite{BFG}, Theorem 4.3]\label{thm:BFG}
Let $\WO$ be the WO-functional on $V_{\hat{\chi}}$ such that  $\WO(\phi^{\circ}) = 1$. For $\lambda \in X_{\ast}(T)$, we have $$\WO(z^{-\lambda}\cdot \phi^{\circ}) = z_1^{-2\lambda_1}\cdot \frac{\A(z^{\rho^{\vee}} z_1^{\lambda_1}(1-q^{-1}z_1^{-1}))}{\A(z^{\rho^{\vee}})}.$$
\end{Theorem}
\begin{remark}
Note that all of the root data here are given with respect to the root system for $\SO(2n+2)$, with $\rho$ denoting the half-sum of the positive roots.
\end{remark}

We will conflate $\hat{\chi}$ with the representation of $\GSp(4)$ of which it is a local lifting (and, hence, we will also conflate the spherical functions for both representations). In the following proof, let $\lambda_1,\lambda_2 \in X_{\ast}(T)$ with $\lambda_1 =  (1,0)$ and $\lambda_2 = (0,1)$.

\begin{proof}[proof of Proposition \ref{prop:shalika}]

We begin by evaluating the two functionals at $z^{\lambda_1}$ and $\pi^{\lambda_1}$, respectively, where $\lambda_1$ (resp.~$\lambda_2$) is embedded in the cocharacter group of the torus of $\SO(6)$ as $(1,0,0)$ (resp.~$(0,1,0)$). In this case, we have that $$\WO(z^{-\lambda}\cdot \phi^{\circ}) = \frac{\A(z^{\rho^{\vee}}z_1)-q^{-1}\A(z^{\rho^{\vee}})}{\A(z^{\rho^{\vee}})}.$$ In order to give explicit expressions for these alternators, we will need to make a choice of quasicharacters $\mu_i$ such that $\mu_i^2 = \chi_i$ for each $i$ - there are two options for each $\mu_i$, and we make one arbitrarily. Let $\xi_i = \mu_i(\pi)$ for each $i$. We find that 
\begin{align*}
\A(z^{\rho^{\vee}}) = \A(\xi_1^3\xi_2) &= \frac{(\xi_1^2 + 1)(\xi_1\xi_2 + 1)(\xi_1\xi_2 - 1)(\xi_2^2 + 1)(\xi_1 + \xi_2)(\xi_1 - \xi_2)}{\xi_1^3\xi_2^3}, \textup{ and}
\\ \A(z^{\rho^{\vee}}z_1) = \A(\xi_1^5\xi_2) &= \frac{\xi_1^4\xi_2^2 + \xi_1^2\xi_2^4 - \xi_1^2\xi_2^2 + \xi_1^2 + \xi_2^2}{\xi_1^2\xi_2^2}\cdot\A(\xi_1^3\xi_2).
\end{align*}
Simplifying, we see that \begin{equation}\label{wospherical} \WO(z^{-\lambda}\cdot \phi^{\circ}) = \frac{\xi_1^4\xi_2^2 + \xi_1^2\xi_2^4 - \xi_1^2\xi_2^2 + \xi_1^2 + \xi_2^2 - q^{-1}\xi_1^2\xi_2^2}{\xi_1^2\xi_2^2} = z_1+z_2-1+z_2^{-1}+z_1^{-1}-q^{-1}.\end{equation}

On the other hand, in order to calculate $\calF(\pi^{-\lambda}\cdot 1_{T(\goth{o})UK})$, where $\calF$ is the functional from the universal principal series $M$ to $V_{\sigma}$ defined in the diagram \eqref{diagram}, we can use Theorem \ref{thm:alternator} with $\Phi_{-1}^{+} = \{\alpha_1,\alpha_1+\alpha_2\}$ and $\Phi_{q}^{+} = \{\alpha_2,2\alpha_1+\alpha_2\}$, along with the commutativity of \eqref{diagram}. Hence, we have that 
\begin{align*}
\calF(\pi^{-\lambda}\cdot 1_{T(\goth{o})UK}) &= \sum_{w \in W} \goth{T}_w \cdot \pi^{\lambda}
\\&= N\cdot \frac{\pi^{2\lambda_1+\lambda_2} + \pi^{\lambda_1+2\lambda_2} - q^{-1}\pi^{\lambda_1+\lambda_2} - \pi^{\lambda_1+\lambda_2} + \pi^{\lambda_1}+ \pi^{\lambda_2}}{\pi^{\lambda_1+\lambda_2}},
\\\textup{where } N &= \frac{-(q^{-1} + 1)(q^{-1}\pi^{\lambda_2} - \pi^{\lambda_1})(\pi^{\lambda_1+\lambda_2} - q^{-1})}{\pi^{2\lambda_1+\lambda_2}}.
\end{align*} As defined in \eqref{diagram}, $\calF$ is not normalized so that $\calF(1_{T(\goth{o})UK}) = 1$, as WO is in Theorem \ref{thm:BFG}. Indeed, we see that 
\begin{align*}
\calF(1_{T(\goth{o})UK}) &= \sum_{w \in W} \goth{T}_w \cdot 1
\\&= N.
\end{align*}
So, if we normalize $\calF$ so that $\calF(1_{T(\goth{o})UK}) = 1$, we see that $$\calF(\pi^{-\lambda}\cdot 1_{T(\goth{o})UK}) = \pi^{\lambda_1} + \pi^{\lambda_2}-1+\pi^{-\lambda_2}+\pi^{-\lambda_1}-q^{-1},$$ which agrees with \eqref{wospherical}, indicating that the WO-functional is a lift of the proposed intertwiner corresponding to $\sigma$.

\end{proof}

\section{Unique Models and the Springer Correspondence} \label{sec:springer}

In this section, we will describe how we expect to construct a gGGr containing $V_{\tau}$ in its $J$-fixed vectors for a given irreducible representation $\tau$ of $\calH_0$.

As described in Section \ref{sec:intro}, the trivial character and the sign character of $\calH_0$ are connected to the spherical model and the Whittaker model, respectively, and the character $\ve$ of $\calH_0$ that acts by $-1$ on long simple roots and by $q$ on short simple roots is similarly connected to the Bessel model for $G = \SO(2n+1)$ or $G = \Sp(2n)$. As mentioned in Section \ref{sec:gggr}, we believe that the Springer correspondence plays a major role in this connection.






The Springer correspondence is a bijection between irreducible representations of $W$ and pairs $(\calO, \mu)$, where $\calO$ is a nilpotent orbit of the Lie algebra and $\mu$ is a character of $A(\calO)$, a subgroup of the $G$-equivariant fundamental group. Geometrically, this bijection arises from the realization of the irreducible representations of $W$ in the top degree cohomology group of partial flag varieties. If $G$ is defined over a finite field,  Kawanaka suggests in Conjecture 2.4.5 in \cite{Kaw} that $V_{\tau}$ should appear, with multiplicity one, in the $B$-fixed vectors of the gGGr $\Gamma_{A,\alpha}$, where the orbit containing $A$ is associated to $\tau$ primarily using the Springer correspondence. Taking inspiration from Kawanaka's conjecture, it is believed that the analogous picture for $G$ defined over a $p$-adic field is the following:
$$\begin{tikzpicture}[>=angle 90]
\matrix(a)[matrix of math nodes,
row sep=3em, column sep=4em,
text height=1.5ex, text depth=0.25ex]
{\ind_B^G(\chi_{\textup{univ}}^{-1})^{J} & \Gamma_{A,\alpha}^J \\& V_{\tau} \simeq R\\};
\path[->,font=\scriptsize]
(a-1-1) edge node[above]{$\calF$}
(a-1-2);
\path[->,font=\scriptsize]
(a-1-2) edge node[right]{(evaluate at 1)}
(a-2-2);
\path[dashed,->,font=\scriptsize]
(a-1-1) edge node[below]{$\calF_1$}
 (a-2-2);
\end{tikzpicture}$$
In this diagram, $\calF$ is a left $\calH$-intertwiner of the universal principal series and the gGGr $\Gamma_{A,\alpha}$, and $\calF_1$ is the functional obtained by evaluation at 1, i.e.~$\phi \mapsto \calF(\phi)(1) \in R$, where $\phi \in \ind_B^G (\chi_{\textup{univ}}^{-1})^J$.

It should be noted that $A$ is not simply the nilpotent orbit associated to $\tau$ under the bijection described by the Springer correspondence - as we will describe below, the connection between $A$ and $\tau$ goes a bit deeper than this. We emphasize here that the exact nature of this connection is still under investigation, in part due to the limited number of data points currently available - we hope to find explicit examples that fit into this program beyond the three mentioned above. 



In what follows, we will restrict ourselves to the setting of $\GSp(2n)$. In this case, we can regard the Springer correspondence as a combinatorial recipe between the two relevant sets, so that we can quickly get to the heart of our proposed connection between gGGr's and characters of $\calH_0$.  Recalling that for type $C_n$, $W$ is the semidirect product of $S_n$ and $(\Z/2)^n$, we can parametrize these representations as well as the nilpotent orbits of $\sp(2n)$ using the following theorems, which can be found in \cite{CM}:


\begin{Theorem}[\cite{CM}, Theorem 10.1.2]
The irreducible representations of the Weyl group $W$ of type $C_n$ are parametrized by ordered pairs $({\bf p},{\bf q})$ of partitions such that $|{\bf p}| + |{\bf q}| = n$. The resulting representation has dimension $$\dim\pi_{({\bf p},{\bf q})} = {n \choose |{\bf p}|}(\dim \pi_{{\bf p}})(\dim \pi_{{\bf q}}).$$ We also have $$\pi_{(\overline{{\bf p}},\overline{{\bf q}})} \simeq \pi_{({\bf q},{\bf p})} \otimes sgn,$$ where $\overline{{\bf p}}$ denotes the conjugate partition of ${\bf p}$, and $sgn$ denotes the sign character. The representation $\pi_{({\bf p},{\bf q})}$ is characterized by the following property. Let V be the subspace of $\pi_{({\bf p},{\bf q})}$ consisting of all vectors on which the first $|{\bf p}|$ copies of $\Z/2$ act trivially while the remaining $|{\bf q}|$ copies act by $-1$. Then $S_{|{\bf p}|} \times S_{|{\bf q}|}$ acts on $V$ according to the representation $\pi_{{\bf p}}\times \pi_{{\bf q}}$.
\end{Theorem}

Using this parametrization, we can see that the four characters of $\calH_0$ correspond to the pairs of partitions $([n],\emptyset)$, $(\emptyset,[1^n])$, $(\emptyset,[n])$, and $([1^n],\emptyset)$. Using Tits' deformation theorem, we can show that these first two partitions correspond to the trivial and sign characters, respectively. The ordered pair $(\emptyset,[n])$ corresponds to the character of $\calH_0$ acting by $q$ on short simple roots and $-1$ on long simple roots, and $([1^n],\emptyset)$ corresponds to the character acting by $-1$ on short simple roots and $q$ on long simple roots. 

For the nilpotent orbits and their corresponding component groups, we find that we have the following parametrizations in type $C_n$:

\begin{Theorem}[\cite{CM}, Theorem 5.1.3]\label{thm:orbits2}
Nilpotent orbits in $\sp(2n)$ are in one-one correspondence with the set of partitions of $2n$ in which odd parts occur with even multiplicity.
\end{Theorem}

\begin{Theorem}[\cite{CM}, Corollary 6.1.6]\label{thm:orbits}
$$A(\calO_{{\bf d}}) = \left\{\begin{array}{ll} (\Z/2)^b & \textup{if all even parts have even multiplicity} \\ (\Z/2)^{b-1} & \textup{otherwise,}\end{array}\right.$$ where $b$ is the number of distinct nonzero parts of ${\bf d}$.
\end{Theorem}

In the case of $\sp(4)$, this means we have four nilpotent orbits corresponding to the partitions $[4],[2^2],[2,1^2]$, and $[1^4]$, and for each of these partitions, the group $A(\calO_{{\bf d}})$ is trivial except for the partition $[2^2]$, for which $A(\calO_{{\bf d}}) \simeq Z/2$. Using these parametrizations, the Springer correspondence gives us the following associations between the irreducible representations of $W$ and pairs of nilpotent orbits and characters of the component group:\\

\begin{center}
\begin{tabular}{|l|l|}
\hline
$({\bf p},{\bf q})$ & $({\bf d},\mu_{\calO})$ \\
\hline
$(\emptyset,[1^2])$ & $([1^4], 1)$ \\
\hline
$([2],\emptyset)$ & $([4], 1)$ \\
\hline
$([1^2],\emptyset)$ & $([2,1^2], 1)$ \\
\hline
$(\emptyset,[2])$ & $([2^2], 1)$ \\
\hline
$([1],[1])$ & $([2^2], sgn)$ \\
\hline
\end{tabular}
\end{center}

If we use Kawanaka's gGGr construction to build the Whittaker functional, we see that the nilpotent element $A$ that we start with lives in the orbit $[4]$ according to Theorem \ref{thm:orbits}. If we do the same thing for the spherical functional, we see that we begin with an element of the orbit $[1^4]$. However, Brubaker, Bump, and Licata showed in \cite{BBL} that the Whittaker functional is an intertwiner for the sign character of $\calH_0$, which is associated to $[1^4]$ via the Springer correspondence, and Brubaker, Bump, and Friedberg showed in \cite{BBF} showed that the spherical functional is an intertwiner for the trivial character of $\calH_0$, which is associated to $[4]$ via the Springer correspondence. This led to the conjecture that, if one starts with an irreducible representation of $\calH_0$, then one should be able to construct a gGGr in which this representation is realized with multiplicity one from an element of the nilpotent orbit whose associated partition is the transpose of the partition associated to the nilpotent orbit associated to the original $\calH_0$-representation via the Springer correspondence. In the example established in this paper, we see that the Bessel functional is associated to the character $\ve$ of $\calH_0$, which, via the Springer correspondence is associated to the nilpotent orbit parametrized $[2,1^2]$. However, the transpose of this partition is $[3,1]$, which does not appear in the parametrization of the nilpotent orbits of $\sp(4)$ given in Theorem \ref{thm:orbits2}. The issue here is that, while, for type $A_{n-1}$, the transpose is an order-reversing involution of the Hasse diagram for the nilpotent orbits of $\sl(n)$, the analogous involution for type $C_n$ is a bit more complicated. In particular, if the partition ${\bf d}$ is associated to a given nilpotent orbit, but ${\bf d}^{\top}$ is not associated to any nilpotent orbit, then we follow further instructions in \cite{CM} for how to manipulate ${\bf d}^{\top}$ in order to find the image of ${\bf d}$ in the order-reversing involution; these manipulations are referred to as the $C$-collapse of ${\bf d}^{\top}$. In the case of the partition $[3,1] = [2,1^2]^{\top}$, the $C$-collapse of this partition is $[2^2]$, which is exactly the orbit containing our original nilpotent element $A$ in Section \ref{sec:besselmodel}.

When we generalize this conjecture to $\sp(2n)$, we see that it correctly associates the sign character with $[2n]$ and the trivial character with $[1^{2n}]$, but that it incorrectly associates $\ve$ with $[2,1^{2n-2}]$; we know from Section \ref{sec:besselmodel} that the gGGr whose $J$-fixed vectors contain $\ve$ is constructed from a representative from the orbit $[2^n]$. With this in mind, we now believe that the path from an irreducible representation of the Hecke algebra to its associated gGGr goes through the Langlands dual group, $^LG$, of $G$ (recall that both $G$ and $^LG$ have the same Weyl group, so having this correspondence go through the dual group versus through $G$ is not something that would be detectable from $\calH$). Inspired by \cite{Ginz}, our idea is that, in order to determine from which nilpotent orbit $A$ should be chosen, we start with an irreducible representation $\tau$ of $\calH_0$ and apply the Springer correspondence to $^LG$ to get the pair $({\bf d}, \mu)$. We then take $\overline{\bf d}$ to be the image of ${\bf d}$ under the appropriate order-reversing involution, $\iota$, of the set of nilpotent orbits, and pick $A$ from the special orbit of $G$ corresponding to $\overline{\bf d}$ under the bijection, $\beta$, between the set of special nilpotent orbits of $G$ and the set of special nilpotent orbits of $^LG$. In types $B_n$ and $C_n$, a special nilpotent orbit is simply a nilpotent orbit ${\bf d}$ such that ${\bf d}^{\top}$ is also a nilpotent orbit. The partial order on the set of nilpotent orbits is defined as follows: geometrically, if $\calO$ and $\calO'$ are two nilpotent orbits, then we say that $\calO \leq \calO'$ if $\overline{\calO} \subset \overline{\calO}'$, where $\overline{\calO}$ refers to the Zariski closure of $\calO$; translated to our parametrizations, we have that ${\bf d} \leq {\bf d}'$ if $$\sum_{1 \leq j \leq k} d_j \leq \sum_{1 \leq j \leq k} d_j' \textup{ for $1 \leq k \leq N$},$$ where ${\bf d} = [d_1,\hdots, d_N]$, ${\bf d}' = [d_1',\hdots, d_N']$ are partitions of $N$. This partial order on partitions is referred to as \emph{dominance order}. The emphasis that is placed on special orbits in this program is due to a result of M\oe glin in \cite{Moeg} regarding the Fourier coefficients of smooth, irreducible representations of $G$ - to summarize, if $\pi$ is such a representation, then the Fourier coefficients of $\pi$ are associated to unipotent orbits. Looking at the set of unipotent orbits associated to the non-zero Fourier coefficients of $\pi$, M\oe glin proved in Theorem 1.4 in \cite{Moeg} that the maximal orbits in this set were special orbits.

We end this paper by offering a formal conjecture regarding the connection between the characters of $\calH_0$ and the gGGr's, along with the evidence that has been compiled so far that supports this conjecture. In the following conjecture, $G$ will denote a split, connected reductive group over $F$.

\begin{Conjecture}\label{conj:springer}
Let $\tau$ be a linear character of $\calH_0$, and let ${\bf d}$ be the nilpotent orbit of $^LG$ associated to $\tau$ via the Springer correspondence and Tits' deformation theorem. Let ${\bf d}'$ be the special nilpotent orbit ${\bf d}' := \beta(\iota({\bf d}))$ of $G$. If $A$ is a representative of ${\bf d}'$, then the gGGr $\Gamma_{A,1}$ (as defined in Section \ref{sec:intro}) is a model for the universal principal series $\ind_B^G\chi_{\textup{univ}}^{-1}$ satisfying the diagram \eqref{diagram}.
\end{Conjecture}

\begin{remark}
Note that the representation $\mu$ in the pair $({\bf d}, \mu)$ associated to $\tau$ via the Springer correspondence does not play a significant role in identifying $A$.
\end{remark}

\begin{remark}
For $G = \Sp(4)$, there are only three special orbits, implying that the spherical, Whittaker, and Bessel models are the only models that are realized via this program in this case. In particular, this case can be considered degenerate, as we cannot associate a gGGr to the remaining character of $\calH_0$.
\end{remark}

As an example, consider the case where $G = \Sp(6)$, whence $^LG = \SO(7)$. In this case, we have the following list of special orbits, listed according to the partial order described above:
\begin{equation}\label{table:orbits}
\begin{tabular}{|l|l|}
\hline
$\Sp(6)$ & $\SO(7)$ \\
\hline
$[6]$ & $[7]$ \\
\hline
$[4,2]$ & $[5,1^2]$ \\
\hline
$[3^2]$ & $[3^2,1]$ \\
\hline
$[2^3]$ & $[3,2^2]$ \\
\hline
$[2^2,1^2]$ & $[3,1^4]$ \\
\hline
$[1^6]$ & $[1^7]$ \\
\hline
\end{tabular}
\end{equation}

The bijection between orbits of $G$ and $^LG$ mentioned above is simply the one suggested by the partial ordering, which is depicted in Table \eqref{table:orbits}. Thus, according to Conjecture \ref{conj:springer}, we see that $\ve$ corresponds to the pair $([3^2,1],1)$ for $^LG = \SO(7)$. Since $[3^2,1]$ is a special orbit, its transpose $[3,2^2]$ is its image under the usual order-reversing involution, and we see that $[3,2^2]$ corresponds to $[2^3]$ under the bijection between special nilpotent orbits of $G$ and special nilpotent orbits of $^LG$, as desired.

We also point out that the trivial character of $\calH_0$ still corresponds to $[1^{2n}]$ under this new conjecture, and the sign character still corresponds to $[2n]$. One can check that this new conjecture is also compatible with our results in this paper regarding $\ve$ for $n=2$. Additionally, one can check that this conjecture is compatible with the results of \cite{BBF2}, in which $G = \SO(2n+1)$ and $^LG = \Sp(2n)$.

\newpage
\bibliographystyle{amsplain}
\bibliography{bessel}
\end{document}